\documentclass[leqno,12pt]{amsart} 
\usepackage{amsmath, amssymb}
\usepackage{mathtools}
\usepackage{amsthm} 
\usepackage[all]{xy}
\usepackage{enumitem}
\usepackage{color}	
\usepackage{hyperref} 
\usepackage{mathrsfs} 
\usepackage{graphicx}
\usepackage{cleveref}
\usepackage[svgnames]{xcolor}
\usepackage{tikz}
\usepackage{pagecolor}

\renewcommand{\thefootnote}{} 

\AtBeginDocument{
   \def\MR#1{}
}

\theoremstyle{plain} 
\newtheorem{theorem}{\indent\sc Theorem}[section]
\newtheorem{lemma}[theorem]{\indent\sc Lemma}
\newtheorem{corollary}[theorem]{\indent\sc Corollary}
\newtheorem{proposition}[theorem]{\indent\sc Proposition}

\newtheorem{theoremIntro}{\indent\sc Theorem}

\newtheorem{corollaryIntro}[theoremIntro]{\indent\sc Corollary}

\theoremstyle{definition} 
\newtheorem{definition}[theorem]{\indent\sc Definition}
\newtheorem{remark}[theorem]{\indent\sc Remark}
\newtheorem{example}[theorem]{\indent\sc Example}

\crefname{theorem}{Theorem}{Theorems}
\crefname{lemma}{Lemma}{Lemmata}
\crefname{corollary}{Corollary}{Corollaries}
\crefname{proposition}{Proposition}{Propositions}
\crefname{remark}{Remark}{Remarks}
\crefname{example}{Example}{Examples}
\crefname{section}{Section}{Sections}
\crefname{definition}{Definition}{Definition}
\crefname{equation}{}{}
\crefname{claim}{Claim}{Claims}
\crefname{theoremIntro}{Theorem}{Theorems}
\crefname{corollaryIntro}{Corollary}{Corollaries}

\newcommand{\C}{\mathbb{C}}
\newcommand{\R}{\mathbb{R}}
\newcommand{\Z}{\mathbb{Z}}
\newcommand{\N}{\mathbb{N}}

\newcommand{\Hyp}{\mathbb{H}}

\def\dim{\mathop{\mathrm{dim}}\nolimits}

\newcommand{\abs}[1]{\lvert#1\rvert}

\newcommand{\Horo}{\mathcal{H}}

\newcommand{\EG}{\underline{E}G}
\newcommand{\EP}{\underline{E}P}
\newcommand{\Xaug}{X(G,\mathbb{P},d_G)}
\newcommand{\Vaug}{V(G,\mathbb{P},d_G)}
\newcommand{\VaugIR}[1]{V(G,\mathbb{P},d_G;#1^R)}
\newcommand{\VaugIr}[2]{V(G,\mathbb{P},d_G;#1^{#2})}

\newcommand{\barEG}{\overline{\EG}}
\newcommand{\famP}{\mathbb{P}}

\newcommand{\dGP}{\partial (G,\famP)}

\newcommand{\grad}{\mathbf{d}}

\newcommand{\dX}{\partial X}
\newcommand{\dG}{\partial^{\mathrm{bl}} G}

\newcommand{\ciso}{\rho}
\newcommand{\XbarW}{\overline{X}^W}

\newcommand{\cd}[2]{\mathop{\mathrm{cd}}_{#1}(#2)}

\newcommand{\distG}{d_G}
\newcommand{\BM}{$\mathbf{BM}$}
\newcommand{\fBM}{$\mathbf{f}$-$\mathbf{BM}$}

\newcommand{\EN}{EN}
\newcommand{\RDc}{\Rip^{\mathrm{co}}}
\newcommand{\RDp}{\Rip^{\mathrm{pa}}}

\newcommand{\FD}{\mathcal{F}}

\newcommand{\Rip}{R_{D}}
\newcommand{\ERX}{E_{\mathcal{H}}X}
\newcommand{\barERX}{\overline{\ERX}}
\newcommand{\Aut}{\mathrm{Aut}}
\newcommand{\calZ}{$\mathcal{Z}$}

\newcounter{thmenumi}

\makeatletter
\def\address#1#2{\begingroup
\noindent\parbox[t]{7.8cm}{%
\small{\scshape\ignorespaces#1}\par\vskip1ex
\noindent\small{\itshape E-mail address}%
\/: #2\par\vskip4ex}\hfill%
\endgroup}%
\makeatother
\title{\uppercase{Blown-up corona of relatively hyperbolic groups}}
\author{
\textsc{Tomohiro Fukaya} 
}
\date{} 
\begin{document}


\thanks{T.Fukaya was supported by JSPS KAKENHI Grant Number JP19K03471}
\renewcommand{\thefootnote}{\fnsymbol{footnote}} 


\begin{abstract}
We show that under appropriate assumptions, a blown-up corona of a
relatively hyperbolic group is equivariant and the compactification of
the universal space for proper action by the blown-up corona is
contractible. As a corollary, we establish the formula to determine the
covering dimension of the blown-up corona by the cohomological dimension
of the group.
We also show
that the blown-up corona of a hyperbolic group with respect to an almost
malnormal family of quasiconvex subgroups is homeomorphic to the Gromov
boundary of the group.
\end{abstract}

\maketitle

\section{Introduction}
\label{sec:introduction}
Let $G$ be a hyperbolic group. There exists a $G$-finite simplicial
complex $\EG$ which is a universal space for proper action \cite[Theorem
4.12.]{MR1086648}. Let $\barEG$ be the Gromov compactification of
$\EG$. Bestvina and Mess~\cite{MR1096169} showed that $\overline{\EG}$ is
contractible. 

\begin{theorem}$($\cite[Theorem 1.2.]{MR1096169}$)$
\label{thm:BM-hypgr} Let $G$ be a hyperbolic group, and let $\EG$ be a
$G$-finite simplicial complex which is a universal space for proper
action. Then the Gromov compactification $\barEG$ is absolute retract (AR),
and the boundary $\partial G = \barEG\setminus \EG$ is a
$\mathcal{Z}$-set. Especially, $\barEG$ is contractible.
\end{theorem}

Here we recall that a closed subset $Z$ of an ANR space $\bar{X}$ is
called a {\itshape $\mathcal{Z}$-set} if for every
open set $U\subset \bar{X}$, the inclusion $U\setminus Z \subset U$ is a
homotopy equivalence.



In this paper, we show an analogue result of Theorem~\ref{thm:BM-hypgr}
for relatively hyperbolic groups and blown-up corona constructed in
\cite{boundary}.

First, on the existence of the classifying space, the following result
is known.

\begin{theorem}$($\cite[Appendix A.]{relhypgrp}$)$
 Let $G$ be a finitely generated group which is
 hyperbolic relative to a finite family of infinite subgroups
 $\famP=\{P_1,\dots,P_k\}$.  We suppose that each subgroup $P_i$ admits
 a finite $P_i$-simplicial complex $\EP_i$ which is a universal space
 for proper actions. Then $G$ admits a finite $G$-simplicial complex
 $\EG$ which is a universal space for proper actions.
\end{theorem}


Bestvina \cite{Bestv-local-homo-bdry} introduced the \calZ-structures on groups.
The following definition is due to Dranishnikov\cite{Dra-BM-formula}.
\begin{definition}
 Let $G$ be a group. A \calZ-structure on $G$ is a pair $(\bar{X},Z)$
 of compact spaces satisfying the following axioms.
 \begin{enumerate}[label=(\arabic*)]
  \item $\bar{X}$ is AR.
  \item $Z$ is a Z-set in $\bar{X}$.
  \item $X=\bar{X}\setminus Z$ admits a metric that admits a proper cocompact action
        of $G$ by isometries.
  \item The compactification $\bar{X}$ of $X$ is a coarse compactification.
 \end{enumerate}
\end{definition}

We show that, under appropriate assumptions, relatively hyperbolic groups
admits a \calZ-structure.

\begin{theoremIntro}
\label{thmIntro:main-BM} Let $G$ be a finitely generated group that is
 hyperbolic relative to a finite family of infinite subgroups
 $\famP=\{P_1,\dots,P_k\}$.  We suppose that each subgroup $P_i$ admits
 a finite $P_i$-simplicial complex $\EP_i$ which is a universal space
 for proper actions. Let  $\EG$ be a finite $G$-simplicial complex
 which is a universal space for proper actions.

For $i=1,\dots,k$, let $\overline{\EP_i}$ be a coarse compactification
 of $\EP_i$.  Set $W_i\coloneqq\overline{\EP_i} \setminus \EP_i$. Let $\dG$ be
 the blown-up corona of $(G,\famP, \{W_1,\dots, W_k\})$ and let
 $\barEG\coloneqq\EG \cup \dG$ be a coarse compactification.  Then we have,
\begin{enumerate}
 \item \label{item:BM} if each $(\overline{\EP_i},W_i)$ is a \calZ-structure
       of $P_i$, for all $i= 1,\dots,k$, then the pair
       $(\barEG,\dG)$ is a \calZ-structure of $G$, and thus $\dG$ is a
       $\mathcal{Z}$-set in $\barEG$. Especially, $\barEG$ is
       contractible;
 \item \label{item:equiv} if each $\overline{\EP_i}$ is a
       $P_i$-equivariant compactification, for all $i= 1,\dots,k$, then
       $\barEG$ is a $G$-equivariant compactification.
       \setcounter{thmenumi}{\value{enumi}}
\end{enumerate}
\end{theoremIntro}

\begin{remark}
 Dahmani~\cite{MR1974394} proved Theorem~\ref{thmIntro:main-BM}
 for the case where $k=1$, that is, $G$ is a group hyperbolic relative to
 a single subgroup $P$.
\end{remark}

\begin{remark}
 Recently, Takatsu~\cite{takatsu2023blown} and Kikuta independently 
 proved that for K3 surface $X$, 
 the group of automorphisms $\Aut(X)$ is geometrically finite up to finite kernel.
 Thus $\Aut(X)$ admits a \calZ-structure by \cref{thmIntro:main-BM}.
 Takatsu~\cite{takatsu2023blown} also directly
 constructed a \calZ-structure
 for $\Aut(X)$ using the action on the hyperbolic space $\Hyp^n$.
 His argument is only applicable to the groups
 acting on the hyperbolic space. However, it is much simpler than
 the proof of~\cref{thmIntro:main-BM} given in this paper.
\end{remark}

\subsection{Cohomological dimension}
As an application of Theorem~\ref{thm:BM-hypgr}, Bestvina and Mess
established a formula that determines the covering dimension of the
boundary of a torsion free hyperbolic group by the cohomological
dimension of the group. Theorem~\ref{thmIntro:main-BM} generalizes this formula
for relatively hyperbolic groups.
\begin{corollaryIntro}
\label{thm:main-cohomdim} Let $G$ be a finitely generated
 group that is hyperbolic relative to a finite family of infinite
 subgroups $\famP=\{P_1,\dots, P_k\}$.  We suppose that each subgroup
 $P_i$ admits a finite $P_i$-simplicial complex $\EP_i$ which is a
 universal space for proper actions.  For $i=1,\dots,k$, let
 $\overline{\EP_i}$ be a coarse compactification of $\EP_i$.  Set
 $W_i\coloneqq\overline{\EP_i} \setminus \EP_i$.  Let $\dG$ be the
 blown-up corona of $(G,\famP, \{W_1,\dots, W_k\})$. 
Then we have
\begin{align}
\label{eqIntro:cohom-cong}
 H^p (G;\Z G) \cong \widetilde{H}^{p -1}(\dG;\Z)\quad  \text{for all} \quad p\in\N.
\end{align}
Here $\widetilde{H}^\bullet(-)$ denotes the reduced Alexander-Spanier
cohomology.

 We also suppose that 
 $(\overline{\EP_i},W_i)$ is a \calZ-structure of $P_i$
 for all $1\leq i \leq k$. Then we have
\begin{align}
\label{eqIntro:cohomdim}
 \dim(\dG) = \max\{p:H^p(G;\Z G)\neq 0\} -1.
\end{align}
If $G$ is torsion free, then 
\begin{align*}
 \cd{}{G} = \dim(\dG) +1
\end{align*}
where $\cd{}{G}$ notes the cohomological dimension of $G$.
\end{corollaryIntro}
In the case when $\famP$ is empty, the formula~(\ref{eqIntro:cohomdim})
is the results of Bestvina-Mess~\cite{MR1096169}.

\subsection{Blown-up corona of a hyperbolic group with an almost malnormal family of quasiconvex subgroups}

Let $G$ be a group. A finite family $\famP\coloneqq\{P_1,\dots,P_k\}$ of
subgroups of $G$ is said to be {\itshape almost malnormal} if every
infinite intersection of the form $H_i\cap g^{-1}H_jg$ satisfies both
$i=j$ and $g\in H_i$. Bowditch proved the following.
\begin{theorem}
\label{th:hyp-is-relhyp}
 \cite[Theorem 7.11]{MR2922380} Let $G$ be a non-elementary hyperbolic
 group, and let $\famP=\{P_1,\dots, P_k\}$ be an almost malnormal family of
 proper, quasiconvex subgroups of $G$. Then $G$ is hyperbolic relative to
 $\famP$.
\end{theorem}

Manning~\cite{Manning-Bowditch-boundary} gave an alternative proof of
Theorem~\ref{th:hyp-is-relhyp} in terms of the dynamical
characterization of relative hyperbolicity. The main ingredient of his
proof is the explicit description of its Bowditch boundary.
%
%

\begin{theorem}[\cite{Manning-Bowditch-boundary}]
\label{th:G-to-B-full}
 Let $G$ be a hyperbolic group, and let $\famP$ be an almost malnormal
 family of infinite quasiconvex proper subgroups of $G$. Let
 $\mathcal{L}$ be the set of $G$-translates of limit set of elements of
 $\famP$. The Bowditch boundary $\dGP$ is obtained from the Gromov
 boundary $\partial G$ as a decomposition space $\partial G/\mathcal{L}$.
\end{theorem}

Using Manning's description, we obtain the following.
\begin{theoremIntro}
\label{th:B-to-G} Let $G$ be a non-elementary hyperbolic group, and
 $\famP=\{P_1,\dots, P_k\}$ be an almost malnormal family of proper,
 infinite quasiconvex subgroups of $G$. Let $\partial G$ be the Gromov
 boundary of G and let $\partial P_i$ be the limit set of $P_i$ for $i =
 1,\dots,k$.  Let $\dG$ be the blown-up corona of $(G,\famP,\{\partial
 P_1,\dots,\partial P_k\})$. Then $\dG$ is homeomorphic to $\partial G$.
\end{theoremIntro}

Theorem~\ref{th:G-to-B-full} says that if $G$ is a hyperbolic group and
$\famP$ is an almost malnormal family of quasiconvex subgroups, then the
Bowditch boundary $\dGP$ is obtained from the Gromov boundary $\partial
G$. Conversely, Theorem~\ref{th:B-to-G} says that the Gromov boundary is
restored from the Bowditch boundary by blowing up all parabolic points.

\subsection{Organisation of the paper}
In Section~\ref{sec:higs-comp}, we review the definition of the Higson
compactification and the coarse compactification.  In
Section~\ref{sec:suff-cond-contr}, we review the coarse structure on
simplicial complexes, and {\itshape \calZ-structure} of groups,
introduced by Bestvina~\cite{Bestv-local-homo-bdry}.

In Section~\ref{sec:relat-hyperb-groups}, we review the definition of
relatively hyperbolic groups in terms of the combinatorial horoballs and
the augmented space introduced by Groves-Manning~\cite{MR2448064}.

In Section~\ref{sec:contractibitiy}, we give a proof of

Theorem~\ref{thmIntro:main-BM}(\ref{item:BM}). In
Section~\ref{sec:cohom-dimens}, we discuss the cohomological dimension
and give a proof of Corollary~\ref{thm:main-cohomdim}. We also give two
examples of explicit calculations of the covering dimension of the
blown-up coronae of relatively hyperbolic groups.

In Section~\ref{sec:hyp-malnormal}, we review Manning's description
of the Bowditch boundary of a hyperbolic group with an 
almost malnormal family of quasiconvex subgroups, and we give a proof of
Theorem~\ref{th:B-to-G}.

In Section~\ref{sec:act-on-corona}, we construct a group action on the
blown-up corona and give a proof of
Theorem~\ref{thmIntro:main-BM}~ (\ref{item:equiv}).

In \cref{sec:table-symbols}, we list symbols used in this paper.

\subsection*{Acknowledgment}
The author would like to thank Shin-ichi Oguni for his carefully reading
the manuscript and several useful comments.

\section{Coarse space and \calZ-structure}
\label{sec:coarse-space-mathc}
\subsection{Coarse compactification}
\label{sec:higs-comp}
A proper coarse space is a locally compact Hausdorff space 
equipped with a proper coarse structure. See \cite[Definition 2.22]{MR2007488}
for proper coarse structure.
Here we recall the definitions of the Higson compactification and coarse
compactification. For details, see \cite{MR2007488} and
\cite{MR1147350}.
\begin{definition}
\label{def:Higson_function}
 Let $X$ be a proper coarse space. Let  $f\colon X\to \C$ be a
 bounded continuous function. We denote by $\grad f$ the function
\[
 \grad f(x,y) = f(y) - f(x) \colon X\times X \to \C.
\]
We say that $f$ is a {\itshape Higson function}, or, of {\itshape
vanishing variation}, if for each controlled set $E$, the restriction of
$\grad f$ to $E$ vanishes at infinity, that is, for any $\epsilon>0$,
there exists a bounded subset $B$ such that for any $(x,y)\in E
\setminus B\times B$, we have $\abs{\grad f(x,y)}<\epsilon$.
\end{definition}

\begin{remark}
 In the case that $X$ is a proper metric space with a bounded coarse structure,
 the bounded continuous function $f\colon X\to \C$ is a Higson function,
 if and only if, for any $\epsilon> 0$ and $R>0$, there exists a compact
 set $K\subset X$ such that for all $(x,y)\in X\times X\setminus K\times
 K$ with $d(x,y)< R$, we have $\abs{\grad f(x,y)}<\epsilon$.
\end{remark}

The bounded continuous complex valued Higson functions on a proper coarse
space $X$ form an unital
$C^*$-subalgebra of bounded continuous functions on $X$, which we
denote $C_h(X)$. By the Gelfand-Naimark theory, $C_h(X)$ is isomorphic
to a $C^*$-algebra of continuous functions on a compact Hausdorff space.

\begin{definition}
\label{def:Higson-compactification}
 The compactification $hX$ of $X$ characterized by the property 
$C(hX) = C_h(X)$ is called the {\itshape Higson compactification}. Its boundary 
$hX\setminus X$ is denoted $\nu X$, and is called the {\itshape Higson corona} 
of $X$.
\end{definition}

\begin{definition}
 Let $X$ be a proper coarse space, and let $\overline{X}$ be a
 compactification of $X$. We say that $\overline{X}$ is a 
 \emph{coarse compactification} if the identity map $X\to X$ extends to
 a continuous surjection $hX\to \overline{X}$. 
 We call $\overline{X}\setminus X$ a \emph{corona} of $X$.
\end{definition}

The one point compactification of a proper coarse space is a coarse
compactification. If $X$ is a proper geodesic Gromov hyperbolic space,
then the Gromov compactification $\overline{X}$ is a coarse
compactification.

\subsection{Coarse metric simplicial complex}
\label{sec:suff-cond-contr}
In this section, we construct a coarse structure for a simplicial complex.
For details of coarse structure, see \cite{MR2007488}. 

\begin{definition}
\label{def:cm-cpx-0-map}
A {\itshape coarse metric simplicial complex} is a simplicial complex $X$ equipped with the coarse structure defined as follows.

We denote by $X^{(0)}$ the set of vertices
 of $X$. We also denote by $X^{(1)}$ the 1-skeleton of $X$. We equip
 $X^{(0)}$ with a proper metric $d_{X^{(0)}}$ induced by the graph structure of $X^{(1)}$,
 where we consider that each edge has length 1.
We choose a map $V\colon X\to X^{(0)}$ such that
\begin{itemize}
 \item for $p\in X$, the image $\varphi_V(p)$ is a vertex of a
       simplex of $X$ which contains $p$;
 \item the restriction to $X^{(0)}$ is the identity.
\end{itemize} 
We say that $V$ is a {\itshape 0-skeleton map}.
We remark that $V$ is pseudocontinuous. 
We equip $X^{(0)}$ with the bounded coarse structure induced by $d_{X^{(0)}}$.
Then we equip $X$ with the pullback
coarse structure induced by $V$. Then $X$ is a proper coarse space.
\end{definition}
See \cite[Section 2]{boundary} for the pseudocontinuous map and the
pullback coarse structure.


\begin{example}
\label{ex:c-mtrc-s-cpx} Let $G$ be a finitely generated group and let
$d_G$ be a word metric on $G$. We equip $G$ with a bounded coarse
structure induced by $d_G$. We suppose that there exists a finite
$G$-simplicial complex $\EG$ which is a universal space for proper
action. Then $\EG$ is a coarse simplicial complex which is coarsely
equivalent to $G$.
\end{example}

\subsection{\calZ-structure and the condition \BM{}}

\begin{definition}
\label{def:BMconditions}
 Let $X$ be a coarse metric simplicial complex and let $\bar{X}$ be a
 coarse compactification of $X$. Set $\partial X\coloneqq \bar{X}\setminus X$. 
\begin{enumerate}
 \item Let
 $x\in \partial X$ be a point and let $U$ be a neighbourhood of $x$ in
 $\bar{X}$. We say that a neighbourhood $V$ of $x$ in $\bar{X}$ is
 weakly contractible in $U$ if every finite subcomplex $K$ in
 $V\cap X$ is contractible in $U\cap X$.
 \item Let $x\in \partial X$ be a point. We say that the triplet
$(\bar{X}, X, x)$ satisfies local-\BM{}, if for every neighbourhood $U$
       of $x$ in $\bar{X}$, there exists a 
       neighbourhood $V$ of $x$ which is weakly contractible in $U$.
 \item \label{defIntro:BM}
       We say that the pair $(\bar{X},X)$ satisfies the condition \BM{} if
       for every point $x\in \partial X$, the triplet $(\bar{X},X,x)$ satisfies 
       local-\BM{}.
 \item We say that the pair
       $(\bar{X},X)$ satisfies the condition \fBM{} if $(\bar{X},X)$ satisfies
       \BM{} and $\bar{X}$ has a finite covering dimension.
\end{enumerate}
\end{definition}



\begin{lemma}
\label{lem:map-local-BM}
 Let $X$ be a coarse metric simplicial complex and let $\bar{X}$ be a
 coarse compactification of $X$. Set $\partial X\coloneqq \bar{X}\setminus X$. 
 Let $x\in \partial X$ be a point. Suppose
 that there exist
\begin{itemize}
 \item a neighbourhood $U_x$ of $x$ in $\bar{X}$;
 \item a coarse metric simplicial complex $Y$;
 \item a compactification $\bar{Y}$ of $Y$; 
 \item a map $f \colon U_x \to \bar{Y}$, such that $f(U_x\setminus
       X)\subset \bar{Y}\setminus Y$ and $f$ is a homeomorphism onto the
       image.
\end{itemize}
Then the triplet $(\bar{X},X,x)$ satisfies local-\BM{} if and only if so does
the triplet $(\bar{Y},Y,f(x))$.
\end{lemma}

\begin{lemma}
\label{lem:hmtpymap-local-BM}
 Let $X$ be a coarse metric simplicial complex and let $\bar{X}$ be a
 coarse compactification of $X$. Set $\partial X\coloneqq \bar{X}\setminus X$. 
 Let $x\in \partial X$ be a point. Suppose
 that there exist
\begin{itemize}
 \item a neighbourhood $U_x$ of $x$ in $\bar{X}$;
 \item a coarse metric simplicial complex $Y$;
 \item a compactification $\bar{Y}$ of $Y$; 
 \item a map $f \colon U_x \to \bar{Y}$ such that
       $f(U_x\setminus X)\subset \bar{Y}\setminus Y$,
       and $f(U_x)$ is a neighbourhood of $f(x)$;
 \item a map $g\colon f(U_x) \to \bar{X}$;
 \item a homotopy $h\colon U_x\cap X \times [0,1] \to X$ such that
       $h(x,0) = x,\, h(x,1) = g\circ f(x)$ for all $x\in U_x\cap X$
       and $h$ is close to the projection $U_x\cap
       X\times [0,1]\ni (x,t) \mapsto x$.
\end{itemize}
Then the triplet $(\bar{X},X,x)$ satisfies local-\BM{} if so does the
triplet $(\bar{Y},Y,f(x))$.
\end{lemma}
\begin{proof}
 Let $U$ be a neighbourhood of $x$ in $\bar{X}$. Set $U'\coloneqq U\cap U_x$. 
 Since $h$ is close to the projection onto the first coordinate, we
 can find a smaller neighbourhood $U''\subset f(U')$ of $f(x)$ such that
 $h(f^{-1}(U'')\cap X\times [0,1]) \subset U$. Especially, we have
 $g(U''\cap Y)\subset U$.
 
Suppose that the triplet $(\bar{Y},Y,f(x))$ satisfies
 local-\BM{}. Then there exists a neighbourhood $V''$ of $f(x)$ which is
 weakly contractible in $U''$. 

 Set $V\coloneqq f^{-1}(V'')$. Let $K$ be a simplicial complex and let
 $\eta\colon K\to X$ be a simplicial map such that $\eta(K) \subset
 V$. Then there exists a homotopy $h'\colon K\times[1,2]\to U''$ such
 that $h'(q,1) = f \circ \eta(q),$ and $h'(-,2)$ is a constant map.

 We define a homotopy $H\colon K\times [0,2] \to U$ as follows. For
 $0\leq t\leq 1$, set $H(q,t)\coloneqq h(\eta(q),t)$. For $1\leq t\leq 2$, set
 $H(q,t)\coloneqq g\circ h'(q,t)$. Then we have that 
 $H(q,0)= \eta(q)$, and  $H(-,2)$ is a constant map.
\end{proof}

The following proposition follows from the result by
Bestvina-Mess~\cite[Proposition 2.1.]{MR1096169}.
\begin{proposition}
\label{prop:Z-set} Let $X$ be a coarse metric simplicial complex and let
 $\bar{X}$ be a coarse compactification of $X$.
 Set $\partial X\coloneqq \bar{X}\setminus X$. Suppose $(\bar{X},X)$
satisfies \fBM. 
Then $\bar{X}$ is ANR and $\partial X\subset \bar{X}$ is a $\mathcal{Z}$-set.
\end{proposition}


\begin{corollary}
\label{cor:contractible} Let $X$ be a coarse metric simplicial complex
and let $\bar{X}$ be a coarse compactification of $X$.  Suppose
$(\bar{X},X)$ satisfies \fBM{}. 
If $X$ is contractible, then so is $\bar{X}$.
\end{corollary}

\begin{corollary}
 Let $G$ be a group and let $\EG$ be a universal space of proper action.
 Let $\barEG$ be a coarse compactification of $\EG$. Set $\partial G\coloneqq \barEG\setminus \EG$.
 If the pair $(\barEG,\EG)$ satisfies \fBM, then $(\barEG,\partial G)$ 
 is a \calZ-structure of $G$.
\end{corollary}




The following follows from {\cite[Theorem 1]{Dra-BM-formula}}.
\begin{theorem}[Dranishnikov] 
 \label{thm:BM-implies-cohomdim}
 Let $G$ be a group and let $(\bar{X},Z)$ be a \calZ-structure of $G$. 
 Then
 \begin{align*}
  \dim Z = \max\{p:H^p(Z;\Z)\neq 0\}
 \end{align*}
\end{theorem}



For details of the cohomological dimension of groups, see \cite[Chapter VIII]{Brown-cohom-gr}.




\section{Relatively hyperbolic groups}
\label{sec:relat-hyperb-groups}
\subsection{Combinatorial horoball}
Let $G$ be a finitely generated group. Relative hyperbolicity of $G$
is characterized by a properly discontinuous action on a hyperbolic
space in the sense of Gromov. There exist several constructions of such
a space.  Here we review the {\itshape augmented space} defined by
Groves and Manning~\cite[Section 3.2.]{MR2448064}. An advantage of
their construction is that the augmented space is a proper metric space,
so its coarse cohomology is well-defined.

\begin{definition}
\label{def:Horoball}
Let $(P,d)$ be a proper metric space.
{\itshape The combinatorial horoball} based on $P$, denoted by
 $\Horo(P)$, is the graph defined as follows:
\begin{enumerate}
 \item $\Horo(P)^{(0)} = P \times (\N\cup \{0\})$. 
 \item $\Horo(P)^{(1)}$ contains the following two type of edges:
       \begin{enumerate}[label=(\roman*)] 
	\item For each $l\in \N\cup\{0\}$ and $p,q\in P$, 
	      if $0< d(p,q)\leq 2^{l}$
	      then there exists a 
	      {\itshape horizontal edge} connecting $(p,l)$ and $(q,l)$.
	\item For each $l\in \N\cup\{0\}$ and $p\in P$, there exists a 
	      {\itshape vertical edge} connecting $(p,l)$ and $(p,l+1)$.
       \end{enumerate}
\end{enumerate}
For a vertex $(p,l)\in P\times (\N\cup \{0\})$, we define that
the \emph{depth} of $(p,l)$ is $l$.

We endow $\Horo(P)$ with the graph metric. 
\end{definition}

When $P$ is a discrete proper metric space, 
$\Horo(P)$ is a proper geodesic space
 that is hyperbolic in the sense of Gromov. 
See \cite[Theorem 3.8]{MR2448064}. 

Groves and Manning describe the shapes of geodesic segments 
in combinatorial horoballs.

\begin{lemma}[{\cite[Lemma 3.10]{MR2448064}}]
 \label{lem:HoroGeodeGM}
 Let $\Horo(P)$ be a combinatorial horoball.  Suppose that 
 $x,y\in \Horo(P)$ are distinct vertices. 
 Then there is a geodesic $x$ and $y$
 which consists of at most two vertical segments and a single horizontal
 segment of length at most 3. 
 Moreover, any other geodesic between $x$
 and $y$ is Hausdorff distance at most 4 from this geodesic.
\end{lemma}

\begin{lemma}[{\cite[Lemma 3.11]{MR2448064}}]
\label{lem:parabolic-pt-of-horoball} Let $P$ be a proper metric space. We
suppose that $P$ is discrete. Then the Gromov compactification of the
combinatorial horoball $\Horo(P)$ is a one-point compactification of
$P$. Thus the Gromov boundary of $\Horo(P)$ consists of one point,
 called the parabolic point of $\Horo(P)$.
\end{lemma}

\begin{lemma}
\label{lem:HoroGeodeO-A-B}
 Let $P$ be a discrete metric space and let $\Horo(P)$ be the
 combinatorial horoball on $P$. 
 Let $d_{\Horo}$ denote the graph metric on $\Horo(P)$.
 Let $O\coloneqq (o,R),\, A\coloneqq (a,s),\, B\coloneqq (b,t)$
 be vertices of $\Horo(P)$.
 Let $D>0$. We suppose that $R>s+D+2$, $t>s$ and 
 $d_{\Horo}(O,B)>d_{\Horo}(O,A)$.
 Set $A'\coloneqq (a,s+D)$.
 Then we have $d_{\Horo}(A',B)\leq d_{\Horo}(A,B)$.
\end{lemma}

\begin{proof}
 There exists $s',u\in \N$ with $s'\geq R$ such that:
 \begin{enumerate}[label=(\roman*)]
  \item The path $[(o,R),(o,s')]_{V}*[(o,s'),(a,s')]_{H}*[(a,s'),(a,s)]_{V}$
        is a geodesic from $O$ to $A$.
  \item The path $[(a,s),(a,u)]_{V}*[(a,u),(b,u)]_{H}*[(b,u),(b,t)]_{V}$
        is a geodesic from $A$ to $B$.
 \end{enumerate}
 Set $r\coloneqq d_{\Horo}((o,s'),(a,s'))$. 
 If $u<R-2$, then $d_{\Horo}((o,s'),(b,s'))\leq r+1$.
 By considering the path 
 $[(o,R),(o,s')]_{V}*[(o,s'),(b,s')]_{H}*[(b,s'),(b,t)]_{V}$, 
 we have
 \begin{align*}
  d_{\Horo}(O,B) \leq s'-R + r+ 1 + s'-t \leq s'-R + r+ s'-s\leq d_{\Horo}(O,A),
 \end{align*}
This contradicts the assumption. Thus we have $u\geq R-2\geq s+D$. 
 Then we have $d_{\Horo}(A',B)\leq d_{\Horo}(A,B)$.
\end{proof}

\subsection{Augmented space}

Let $G$ be a finitely generated group with a finite family of
infinite subgroups 
$\famP = \{P_1,\dots, P_k\}$. 
We take a finite generating set $\mathcal{S}$ for $G$. 
We assume that
$\mathcal{S}$ is symmetrized, so that $\mathcal{S} = \mathcal{S}^{-1}$.
We endow $G$ with the left-invariant word metric
$d_G$ with respect to $\mathcal{S}$. 

\begin{definition}
\label{def:order-of-cosets}
 Let $G$ and $\famP$ be as above. An order of the cosets of $(G,\famP)$
 is a sequence $\{g_n\}_{n\in \N}$ such that
 $g_i=e$ for $i\in \{1,\dots,k\}$, and 
for each $r \in
\{1,\dots,k\}$, the map 
$\N\to G/P_r: a \mapsto g_{ak+r}P_r$ is bijective. 
Thus the set of all cosets $\bigsqcup_{r=1}^k G/P_r$ is indexed by the map
$\N\ni i \mapsto g_iP_{(i)}$. Here $(i)$ denotes the remainder of $i$ divided by $k$.
\end{definition}

\begin{remark}
\label{rem:order-ij}
 Let $\{g_n\}_{n\in \N}$ be an order of the coset. 
 For $i,j\in \N$, if $g_iP_{(i)} \cap g_jP_{(j)}\neq \emptyset$, 
 then $i=j$.
\end{remark}

We fix an order $\{g_n\}_{n\in\N}$ of the cosets of $(G,\famP)$.
Each coset $g_iP_{(i)}$ has a proper metric $d_i$ which is the
restriction of $d_G$. 
Let $\Gamma$ be the Cayley graph of $(G,\mathcal{S})$.
We denote by $\Horo(g_iP_{(i)};\{0\})$ the full 
subgraph of $\Horo(g_iP_{(i)})$ spanned by
 $g_iP_{(i)}\times \{0\}$.
There exists a natural embedding 
$\psi_i\colon \Horo(g_iP_{(i)};\{0\}) \hookrightarrow \Gamma$ such that
$\psi_i(x,0) = x$ for all $x\in g_iP_{(i)}$. 
\begin{definition}
\label{notation:1}
{\itshape The augmented space } $\Xaug$ is obtained by pasting
 $\Horo(g_iP_{(i)})$
 to $\Gamma$ by $\psi_i$ for all $i\in \N$. 
Thus we can write it as follows:
\begin{align*}
 \Xaug &\coloneqq \Gamma\cup \bigcup_{i\in \N} \Horo(g_iP_{(i)}).
\end{align*}
\end{definition}

\begin{remark}
\label{rem:vertices}
The vertex set of $\Xaug$ can naturally be identified with 
the disjoint union of $G$ and the set of 
triplet $(i,p,l)$, where $i\in \N$, $p\in g_iP_{(i)}$, and $l\in \N$.
We sometimes denote $g\in g_iP_{(i)}$ by $(i,g,0)$ for simplicity.
\end{remark}

\begin{definition}
A group $G$ is hyperbolic relative to $\famP$ if the augmented space
 $\Xaug$ is hyperbolic in the sense of Gromov. A member of $\famP$ is
 called a \emph{peripheral subgroup}.
\end{definition}

In Section~\ref{sec:example-cohomdim}, we describe two examples of
relatively hyperbolic groups.

\section{\calZ-structure of relatively hyperbolic group}
\label{sec:contractibitiy}
\subsection{Setting}
\label{subsec:setting} Throughout this section, we fix the following
setting.  Let $G$ be a finitely generated group which is hyperbolic
relative to a finite family of infinite subgroups
$\famP=\{P_1,\dots, P_k\}$. We suppose that each $P\in \famP$ has an infinite
index in $G$. We choose an order of coset $\{g_n\}_{n\in\N}$ of
$(G,\famP)$ in the sense of Definition~\ref{def:order-of-cosets}. We
choose a finite generating set $\mathcal{S}$ of $G$.  Let $\distG$ be a
word metric for $\mathcal{S}$. Let $\Gamma$ be a Cayley
graph of $(G,\mathcal{S})$. Let $\Xaug$ be an augmented space and
let $\Vaug$ be its vertex set.  We denote by $\dGP$ the Gromov boundary of
$\Xaug$, which is called the {\itshape Bowditch boundary} for the pair $(G,\famP)$

Let $P \in \famP$ be a peripheral subgroup. We equip $P$ with a left
invariant proper metric $d_P\coloneqq d_G|_{P\times P}$.  We suppose that $P$
admits a finite $P$-simplicial complex $\EP$ which is a universal space
for proper actions. 

\begin{definition}
 
 We fix the following notation.
 \begin{itemize}
 \label{def:notationA}
  \item Denote by $V(\EP)$ the vertex set of $\EP$.
  \item Choose a fundamental domain $\Delta_P$ of $P$ on $\EP$ and a base point
        $e_P \in \Delta_P\cap V(\EP)$.
  \item Define a map $\iota_P\colon P\to V(\EP)$ as 
        $\iota_P(g) = g\cdot e_P$ for $g\in P$.
  \item Define a map $V_P\colon \EP \to P$ as
        $V_P(q)\coloneqq g$ for $q\in \EP$, where $g$ is an element of $P$ such that
        $q\in g\cdot\Delta_P$. 
  \item Equip $\EP$ with the pullback coarse structure
        induced by $V_P$, which is proper, and coarsely equivalent to $P$.
  \item For a subset $K^{0}\subset P$, define 
        $E_P(K^{0}) \coloneqq \cup_{g\in K^{0}} \, (g\cdot\Delta_P)\subset \EP$.
  \item Denote by $s_P$ the parabolic point of $\Horo(P)$ defined in
        Lemma~\ref{lem:parabolic-pt-of-horoball}.
 \end{itemize}
\end{definition}
We suppose that $P$ has a coarse compactification
$\overline{P}$. Set $W_P\coloneqq\overline{P} \setminus P$. Since $P$ and 
$\EP$ are coarsely equivalent, we have a coarse compactification
$\overline{\EP}: = \EP \cup W_P$.

\begin{definition}
\label{def:notationB}
 For each $i\in \N$, we fix the following notations.
 \begin{itemize}
  \item Define a map $\iota_i\colon g_iP_{(i)} \to g_iV(EP_{(i)})$ as 
        $\iota_i(h) = g_i(\iota_{P_{(i)}}(g_i^{-1}h))$ for $h\in g_iP_{i}$.
  \item Define a map $V_i\colon g_i\EP_{(i)} \to g_iP_{(i)}$ as
        $V_P(q)\coloneqq g_i(V_{P_{(i)}}(g_i^{-1}(q)))$ for $q\in g_i\EP_{(i)}$.
  \item For a subset $K^{0}\subset g_iP_{(i)}$, 
        define 
        $E_i(K^{0}) \coloneqq \cup_{g\in K^{0}} \, (g\cdot(\Delta_{P_{(i)}}))
        \subset g_i\EP_{(i)}$.
  \item Set $s_i\coloneqq g_is_{P_{(i)}}$ so that $s_i$ is the parabolic point of
        $\Horo(g_iP_{(i)})$.
 \end{itemize}
\end{definition}

We equip a coset $g_iP_{(i)}$ with 
a left invariant proper metric $d_i\coloneqq d_G|_{g_iP_{(i)}\times g_iP_{(i)}}$.
We set $W_i\coloneqq g_iW_{P_{(i)}}$.  Then we have coarse compactifications
$g_i\overline{P_{(i)}} = g_iP_{(i)} \cup W_i$ and 
$g_i\overline{\EP_{(i)}} = g_i\EP_{(i)} \cup W_i$.

\subsection{Rips complex of the augmented space and universal space}
\label{sec:rips-compl-augm}

We consider the Rips complex $\Rip(\Vaug)$ for positive integer $D$. 
For $r, R\in \N$ with $r+D< R$, we define:
\begin{align*}
 \Vaug_r &\coloneqq \bigsqcup_{i\in \N} (g_iP_{(i)}\times \{n\in \N: n\geq r\});\\
 \Vaug^R &\coloneqq G\sqcup \bigsqcup_{i\in \N} 
             (g_iP_{(i)}\times \{n\in \N: 1\leq n \leq R\});\\
 \Vaug_r^R \coloneqq \bigsqcup_{i\in \N} 
             (g_iP_{(i)}\times \{n\in \N: r\leq n \leq R\})
\end{align*}
We denote the full subcomplexes of $\Vaug_r$, $\Vaug^R$ and $\Vaug_r^R$
in $\Rip(\Vaug)$ by $\Rip(\Vaug)_r$, $\Rip(\Vaug)^R$ and 
$\Rip(\Vaug)_r^R$, respectively.

We consider a $P$-homotopy equivalence 

\begin{align}
\label{eqdef:f_P}
  f_P\colon \Rip^{(n)}(V(P,\{P\},d_P))_r \to \EP
\end{align}

which is defined in the proof of~\cite[Proposition A.3.]{relhypgrp}.
Here $\Rip^{(n)}(V(P,\{P\},d_P))_r$ is the $n$-th barycentric
subdivision of $\Rip(V(P,\{P\},d_P))_r$. For sufficiently large
$n$, we can assume that $f_P$ is a $P$-simplicial map. We can take $n$
independently of $P$ since $\famP$ is a finite family.

We consider a mapping cylinder
\begin{align*}
(R^{(n)}_D(V(P,\{P\},d_P))_r^R\times [0,1])&\cup_{f_P} \underline{E}P.
\end{align*}
Here we identify $(q,1) \in (R^{(n)}_D(V(P,\{P\},d_P))_r^R\times [0,1])$ with
$f_P(q)\in \EP$.
See the proof of~\cite[Proposition A.3.]{relhypgrp} for the detailed
description of this space.

Let $i$ be a positive integer. 
We define 
\begin{align}
\label{def:fi}
   f_i\coloneqq g_i\circ f_{P_{(i)}}\circ g_i^{-1}\colon \Rip^{(n)}(V(g_iP_{(i)},\{P_{(i)}\},d_i))_r \to g_i\EP_{(i)}
\end{align}
From now on, we abbreviate $V(g_iP_{(i)},\{g_iP_{(i)}\},d_i)$ by 
$V[i]$. We also abbreviate 
$\Rip^{(n)}(V[i])$ by $\Rip(V[i])$ to avoid messy notation.
Set 

\begin{align}
\label{eqdef:VaugIRi}
  \VaugIR{i}\coloneqq & G \sqcup \bigsqcup_{j\neq i}(g_jP_{(j)}\times \N) 
 \sqcup (g_iP_{(i)}\times \{n\in\N: 1\leq n\leq R\})\\
\notag
  = & \Vaug \setminus \{(i,g,l):g\in g_iP_{(i)}, l>R\}
\end{align}

Let $\Rip(\VaugIR{i})$ be the full subcomplex of $\VaugIR{i}$ in
$\Rip(\Vaug)$.
We define
\begin{align}
\label{def:ERX}
 \ERX[i]&\coloneqq \Rip(\VaugIR{i}) \cup 
 \left\{\Rip(V[i])_r^R\times[0,1]
 \cup_{f_i} g_i\EP_{(i)} \right\}.
\end{align}
In~(\ref{def:ERX}), we identify
$(q,1)\in \Rip(V[i])_r^R\times[0,1]$ with 
$f_i(q)\in g_i\EP_{(i)}$, and we also identify 
$(q,0) \in \Rip(V[i])_r^R\times[0,1]$
with $q\in \Rip(V[i])_r^R\subset \Rip(\VaugIR{i})$.

We equip $\ERX[i]$ with a coarse structure as follows.
Let 
\[
 V_G\colon \Rip(\Vaug) \to \Vaug
\]
be a 0-skeleton map in the sense of
Definition~\ref{def:cm-cpx-0-map}. We define a map 

\begin{align}
\label{eqdef:mapV_bib}
  V_{[i]}\colon \ERX[i]\to \VaugIR{i}
\end{align}

\begin{itemize}
 \item for $q\in \Rip(\VaugIR{i})$, set $V_{[i]}(q)\coloneqq V_G(q)$;
 \item for $(q,t) \in \Rip(V[i])_r^R\times[0,1)$,
       set $V_{[i]}(q,t)\coloneqq V_G(q)$;	%
 \item for $q\in g_i\EP_{(i)}$, set 
       $V_{[i]}(q)\coloneqq (V_i(q),1)\in g_iP_{(i)} \times \{1\}$,
       where $V_i$ is the map defined in \cref{def:notationB}.
\end{itemize}

We equip $\ERX[i]$ with the pullback coarse structure induced by $V_{[i]}$, which 
is proper, and coarsely equivalent to $\VaugIR{i}$.

Recall that we have a coarse compactification
$g_i\overline{P_{(i)}} = g_iP_{(i)} \cup W_i$.
Then we have a blown-up corona 

\begin{align}
\label{eqdef:dXbib}
  \partial X[i]\coloneqq\dGP \setminus \{s_i\} \sqcup
W_i,
\end{align}

whose topology is defined in \cite[Definition 7.13.]{boundary}. 
Then we have coarse compactifications
\begin{align*}
 \overline{X}[i]  &\coloneqq \VaugIR{i} \cup \partial X[i];\\
 \barERX[i]&\coloneqq \ERX[i] \cup \partial X[i].
\end{align*}

We give an explicit description of a neighbourhood basis of the point of
the corona of $\barERX[i]$.  Set $\hat{P}_i\coloneqq\dGP \setminus \{s_i\}$.
For $x\in \hat{P}_i$, let $U$ be a neighbourhood of $x$ in
$\overline{X}[i]$.  Set $\RDc(U)$ be a union of $U\cap \partial X[i]$
and the full subcomplex of $U\cap \VaugIR{i}$ in $\Rip(\VaugIR{i})$. Then
a neighbourhood basis of $x$ in $\barERX[i]$ consists of the collection
of those $\RDc(U)$.

For $y\in W_i$, let $U$ be a neighbourhood of $y$ in 
$g_i\overline{P_{(i)}} = g_iP_{(i)}\cup W_i$. 
We use the following two maps 
\begin{align*}
 &F_i\colon X^i\to \hat{P}_i; &L_i\colon \hat{P}_i \to g_iP_{(i)}
\end{align*}
which are defined in \cite[Section 7.1.]{boundary}.  Here $X^i\coloneqq
\Gamma\cup \bigcup_{j\neq i}\Horo(g_jP_{(j)})$.  Since $L_i$ is a
coarsely equivalent map, we extend $L_i$ to the map between coarse
compactifications
\[
 \bar{L}_i\colon \hat{P}_i\cup W_i \to g_iP_{(i)}\cup W_i.
\]
Set $U^\circ\coloneqq U\cap g_iP_{(i)}$.
Let $\RDp(U)$ be the union of 
the following subsets.
\begin{enumerate}
 \item \label{item:FinvFUo}
       The full subcomplex of $F_i^{-1}(F_i(U^\circ))$ in 
       $\Rip(\VaugIR{i})$;
 \item The full subcomplex of 
       $\{(x,n)\in U^\circ \times \N: 0\leq n\leq R\}$ in $\Rip(\VaugIR{i})$;
 \item The product $\Rip({U^\circ}_r^R) \times [0,1]$, where 
$\Rip({U^\circ}_r^R)$ is the full subcomplex of\\
       $U^\circ\times\{n\in \N:r\leq n \leq R\}$
       in $\Rip(V[i])_r^R$;
 \item The subcomplex $E_i(U^\circ) \subset g_i\EP_{(i)}$,
       defined in Definition~\ref{def:notationB} ;
 \item The subset $\bar{L}_i^{-1}(U)$.
\end{enumerate}
Then a
neighbourhood basis of $y$ in $\barERX[i]$ consist of the collection of those
$\RDp(U)$.

\subsection{One-point blow up}
\label{sec:proof-theor-refthm:b}
The purpose of this subsection is to prove the following.
\begin{proposition}
\label{prop:BM-single} Suppose $R$ and $D$ are large enough 
$(R> r+D\geq D\geq 4\delta+2)$. 
Let $i\in \N$.
If $(\overline\EP_{(i)}, \EP_{(i)})$
satisfies \BM{}, then so does $(\barERX[i],\ERX[i])$.
\end{proposition}
Here we remark that the pair $(\barERX[i],\ERX[i])$ does not satisfies f-\BM{} 
since $\ERX[i]$ has an infinite dimension.


\begin{lemma}
\label{lem:geodesic-contraction}
 Suppose $D\geq 4\delta + 2$.
 Let $x\in W_i$ and let $V$ be a neighbourhood of $x$ in
 $g_i\overline{P_{(i)}}$.  Let $K$ be a finite simplicial
 complex and $\eta\colon K\to \ERX[i]$ be a simplicial map with
 $\eta(K) \subset \RDp(V)$. Then there exists a homotopy 
 $H\colon K\times [0,1] \to \RDp(V)\cap \ERX[i]$ such that $H(-,0) = \eta(-)$ and 
 $H(-,1)$ is a simplicial map 
\begin{align*}
  \eta'\colon K\to \RDp(V) \cap \left\{\Rip(V[i])_r^R\times[0,1]
 \cup_{f_i} g_i\EP_{(i)} \right\}
\end{align*}
 Moreover, for $p\in K$, if $\eta(p) \in \left\{\Rip(V[i])^R_{r+D})\times[0,1]
 \cup_{f_i} g_i\EP_{(i)} \right\}$, 
 then $H(p,t) = \eta(p)$ for all $t\in [0,1]$.
\end{lemma}

\begin{proof}
The proof is based on the arguments in \cite[Section 4.2.]{MR1086648}.
We denote by $d_X$ the graph metric of the augmented space $\Xaug$.
Set $O_K\coloneqq(i,g_i,R)$, and
\begin{align*}
 K'&\coloneqq \eta^{-1}(\Rip(\VaugIr{i}{r})),\\
 K''&\coloneqq \eta^{-1}(\Rip(\VaugIr{i}{r+D})).
\end{align*}
If $K'$ is empty, we have nothing to do, so we assume that 
$K'\neq \emptyset$.
We choose a vertex $p$ of $K'$
such that 
\begin{align*}
 d_X(O_K,\eta(p)) = \max_{q\in K'^{(0)}}d_X(O_K,\eta(q)), 
\end{align*}
here
$K'^{(0)}$ denotes the set of vertices of $K'$. 
We choose a vertex $\hat{p}$ on a geodesic connecting $O_K$ and
$\eta(p)$ in $\Xaug$ such that
\begin{itemize}
 \item $d_X(\hat{p}, \eta(p)) = [D/2]$;
 \item $d_X(O_K,\hat{p}) = d_X(O_K,\eta(p)) - [D/2]$.
\end{itemize}

We will show that for all vertex $q$ of $K''$, we have
\begin{align}
\label{eq:xKetaq}
 d_X(\hat{p},\eta(q)) \leq 
\max \left\{D, d_X(\eta(p),\eta(q))\right\}.
\end{align}

First, we suppose that $d_X(O_K,\eta(q))\leq d_X(O_K,\eta(p))$.
Since $\Xaug$ is $\delta$-hyperbolic, 
for every vertex $q$ of $K''$, 
we have
\begin{align*}
 d_X(\hat{p},\eta(q)) \leq 
 \max \left\{d_X(O_K,\eta(q)) - d_X(O_K,\eta(p)) 
 + \left[\frac{D}{2}\right] + 2\delta,\,
    d_X(\eta(p),\eta(q)) - \left[\frac{D}{2}\right] + 2\delta\right\}.
\end{align*}
Since $d_X(O_K,\eta(q)) \leq d_X(O_K,\eta(p))$, 
$\left[D/2\right] + 2\delta \leq D$ and 
$2\delta< \left[D/2\right]$, the inequality \cref{eq:xKetaq} holds.

Now we suppose that $d_X(O_K,\eta(q))\geq d_X(O_K,\eta(p))$.
By the choice of $p$, we have $q\in K''\setminus K'$.
Thus the depth of $q$ is strictly larger than that of $p$. So we can apply
\cref{lem:HoroGeodeO-A-B}, and we obtain 
 $d_X(\hat{p},\eta(q)) \leq d_X(\eta(p),\eta(q))$. Therefore 
the inequality~\cref{eq:xKetaq} holds.

For a vertex $q$ of $K$, we set 
\begin{align*}
 \eta_1(q)\coloneqq\begin{cases}
	    \hat{p} & \text{if } \eta(q) = \eta(p);\\
	    \eta(q) & \text{else.}
	 \end{cases}
\end{align*}
We remark that $p$ is not adjacent to any 
vertices in $K\setminus K''$.
Thus the inequality (\ref{eq:xKetaq}) shows that two simplicial maps
 $\eta$ and $\eta_1$ are contiguous, so the geometric realizations of
 these maps are homotopic. Now if 
\begin{align*}
 \eta_1(K)\subset \left\{\Rip(V[i])_r^R\times[0,1]
 \cup_{f_i} g_i\EP_{(i)} \right\},
\end{align*}
then we set $\eta'=\eta_1$ and we have done. Otherwise, there exists a vertex 
$p'\in K'$ such that $\eta_1(p')\in \Rip(\VaugIr{i}{r})$.
We repeat this procedure finitely many times.
\end{proof}

Let $\epsilon$ be a positive number. For a subset $L\subset g_iP_{(i)}$,
we denote by $N(i,L,\epsilon) $ the $\epsilon$-neighbourhood of $L$ in
$g_iP_{(i)}$ with respect to the distance $d_i$.
We also define
\begin{align}
\label{eqdef:EN}
 \EN(i,L,\epsilon)\coloneqq \bigcup_{g\in N(i,L,\epsilon)} g\Delta_{P_{(i)}}.
\end{align}
Here we recall that $\Delta_{P_{(i)}}$ is the fundamental domain of
$\EP_{(i)}$.  Let $\pi_i\colon g_iP_{(i)} \times \N \to g_iP_{(i)}$ be a
projection to the first coordinate.  Let $K$ be a finite subcomplex of
$R(V[i])$.  We denote by $K^{(0)}$ the set of vertices of $K$.

\begin{lemma}
\label{lem:contraction-in-P}
 Let $f_i\colon R(V[i]) \to g_i\EP_{(i)}$ be a 
map defined by (\ref{def:fi}).
There exists a constant 
$\epsilon =\epsilon(R)$ such that for any subcomplex  $K$ of 
$R(V[i])^R$, we have that 
$f_i(K)\subset \EN(i,\pi_i(K^{(0)}),\epsilon)$.
\end{lemma}

\begin{proof} 
We remark that a subgroup $P_{(i)}$ acts cocompactly on 
$R(V[(i)])^R$.
 Let $\FD$ be its fundamental domain. We suppose that $\FD$ contains
a neighbourhood of 
\begin{align*}
 \{(i,e,n): n\in \N, 1\leq n\leq R\}.
\end{align*}
There exists a constant
 $\epsilon = \epsilon(R)$ 
 such that 
 \begin{align}
  g_if_{P_{(i)}}(\FD) \subset \EN(i,\{g_i\},\epsilon).
 \end{align}

Let $K$ be a finite subcomplex of $R(V[i])^R$.
We have
 \begin{align*}
  K \subset \bigcup_{g\in \pi_i(K^{(0)})}g\FD; 
 \end{align*}
Now we have 
\begin{align*}
 f_i(K)&\subset f_i\left(\bigcup_{g\in \pi_i(K^{(0)})}g\FD\right)
 = \bigcup_{g\in \pi_i(K^{(0)})}gf_{P_{(i)}}(\FD)\\
 &\subset  \bigcup_{g\in \pi_i(K^{(0)})}gg_i^{-1}\EN(i,\{g_i\},\epsilon)\\
 &\subset \EN\left(i, \pi_i(K^{(0)}),\epsilon\right).
\end{align*}
\end{proof}


\begin{proposition}
\label{prop:one-pt-co-BM}
Let $x\in \barERX[i]\setminus \ERX[i]$
be a point. If $x\in \dGP\setminus \{s_i\}$, then the triplet 
$(\barERX[i],\ERX[i],x)$ satisfies local-\BM{}.
\end{proposition}

\begin{proof}
 Since $\Xaug$ is hyperbolic in the sense of Gromov, by the proof of
 \cite[Theorem 1.2.]{MR1096169}, the pair
 $(\overline{\Rip(\Vaug))},\Rip(\Vaug))$ satisfies the condition
 \BM{}. Here we denote by $\overline{\Rip(\Vaug))}$ the Gromov
 compactification of $\Rip(\Vaug)$.

 Let $x\in \dGP\setminus \{s_i\}$. It is clear that there
 exists a neighbourhood $U$ of $x$ in $\barERX[i]$ such that $U$ is
 canonically homeomorphic to the corresponding subset of
 $\overline{\Rip(\Vaug)}$. Thus by Lemma~\ref{lem:map-local-BM}, the triplet
 $(\barERX[i],\ERX[i],x)$ satisfies local-\BM{}.
\end{proof}

\begin{proposition}
\label{prop:one-pt-pa-BM}
Suppose that the pair $(\overline{\EP_{(i)}},\EP_{(i)})$ satisfies the
condition \BM.  Let $x\in \barERX[i]\setminus \ERX[i]$
be a point. If $x\in W_i$, then the triplet 
$(\barERX[i],\ERX[i],x)$ satisfies local-\BM{}.
\end{proposition}

\begin{proof}
Let $U$ be a neighbourhood of $x$ in $\barERX[i]$.  There exists a
 neighbourhood $U_1$ of $x$ in $g_i\overline{P_{(i)}} = g_iP_{(i)}\cup
 W_i$ such that $\RDp(U_1)\subset U$. Since the pair
 $(\overline{\EP_{(i)}},\EP_{(i)})$ satisfies the condition \BM, there
 exists a neighbourhood $V_1$ of $x$ in $g_i\overline{P_{(i)}}$ such
 that $E_i(V_1)$ is weakly contractible in $E_i(U_1)$ in the sense of
 Definition~\ref{def:BMconditions}.
Let $\epsilon$ be a constant in Lemma~\ref{lem:contraction-in-P}.
There exists a neighbourhood $V_2$ of $x$ in $g_i\overline{P_{(i)}}$ such that 
$EN(i,V_2\cap g_iP_{(i)},\epsilon)$ is contained in $V_1$.

Set $V\coloneqq \RDp(V_2) \subset U$.
Let $K$ be a finite simplicial complex and let $\eta\colon K\to V$ be 
a simplicial map. By Lemma~\ref{lem:geodesic-contraction}, 
$\eta$ is homotopic to a simplicial map
\[
 \eta'\colon K\to V \cap \left\{\Rip(V[i])_r^R\times[0,1]
 \cup_{f_i} g_i\EP_{(i)} \right\}.
\]


We define a map 
\begin{align*}
 \theta \colon \Rip(V[i])_r^R\times[0,1]
 \cup_{f_i} g_i\EP_{(i)}  \longrightarrow
\Rip(V[i])_r^R\times[0,1]
 \cup_{f_i} g_i\EP_{(i)},
\end{align*}
by $\theta(q,t)\coloneqq f_i(q)$ for $(q,t) \in \Rip(V[i])_r^R\times[0,1)$ and
$\theta(q')\coloneqq q'$ for $q'\in g_i\EP_{(i)}$.
It is clear that $\theta$ is homotopic to the identity.

Set $\eta''\coloneqq \theta\circ\eta'$.
By Lemma~\ref{lem:contraction-in-P}, the image $\eta''(K)$ lies on
$V_1$. Since $V_1$ is weakly contractible in $U_1$,
the map $\eta''$ is homotopic to a constant map in $E_i(U_1)$.
It follows that the map $\eta$ is homotopic to a constant map in 
$U$.
\end{proof}


\begin{proof}[Proof of \cref{prop:BM-single}]
 Proposition~\ref{prop:one-pt-co-BM}
 and Proposition~\ref{prop:one-pt-pa-BM} imply
 Proposition~\ref{prop:BM-single}.
\end{proof}

\subsection{Projective limit}
\label{sec:projective-limit}
Set
\begin{align}
\label{eqdef:ERXn}
 \ERX_n &\coloneqq \Rip(\Vaug)^R \cup 
 \bigcup_{j=1}^{n}
  \left\{\Rip(V[j])_r^R)\times[0,1]
  \cup_{f_j} g_j\EP_{(j)} \right\}\\
\notag
&\cup
\bigcup_{j=n+1}^{\infty}
  \left\{\Rip(V[j]_r^R)\times[0,1]
  \cup \Rip(V[j])_r \times\{1\}\right\}.
\end{align}
In \cite[Appendix A]{relhypgrp}, it is shown that the space
\begin{align*}
 \EG \coloneqq \ERX_\infty =  
  \Rip(\Vaug)^R 
   \cup \bigcup_{j=1}^{\infty}
   \left\{\Rip(V[j]_r^R)\times[0,1]
   \cup_{f_j} g_j\EP_{(j)} \right\}
\end{align*}
is a universal space for proper actions.
For $j\in \N$, we choose a homotopy

\begin{align}
 \label{eqdef:h_j}
  h_j \colon  \Rip(V[j])_r \times
[0,1] \to \Rip(V[j])_r
\end{align}

such that $h_j(-,0)$ is the identity and $h(-,1) = f_j^{-1}\circ f_j$.
Here, $f_j^{-1}$ is a homotopy inverse of $f_j$. 

For $n < m$, 
we define a continuous map $\theta_n^m\colon \ERX_m \to \ERX_n$ 
by the following way.
\begin{itemize}
 \item The restriction to $\Rip(\Vaug)^R$ is the identity.
 \item For $m< j$, the restriction to
       \[
	\Rip(V[j])_r^R\times[0,1] \cup
       \Rip(V[j])_r \times\{1\}
       \] 
       is the identity.
 \item For $n < j \leq m$,
       $\Rip(V[j])_r^R\times [0,1) \ni
       (x,t) \mapsto (h_j(x,t),t)\in \Rip(V[j])_r^R\times [0,1)$.
 \item For $n< j \leq m$, $g_j\EP_{(j)} \ni p 
       \mapsto (f_j^{-1}(p),1) \in  \Rip(V[j])_r\times \{1\}$.
 \item For $1\leq j\leq  n$, the restriction to
       $\Rip(V[j])_r^R\times[0,1]
       \cup_{f_j} g_j\EP_{(j)}$ is the identity.
\end{itemize}

Set $\theta_n \coloneqq \theta_n^\infty\colon \ERX_\infty \to \ERX_n$.
We define a homotopy inverse 
\begin{align}
 \label{eqdef:tau_n}
 \tau_n\colon \ERX_n \to \EG = \ERX_\infty
\end{align}
of
$\theta_n$ as follows.
\begin{itemize}
 \item The restriction to $\Rip(\Vaug)^R$ is the identity.
 \item For $j\in \N$, the restriction to
       $\Rip(V[j])_r^R\times [0,1)$ is the
       identity.
 \item For $j> n$, $\Rip(V[j])_r \times \{1\}\ni (q,1) \mapsto
       f_j(q) \in g_j\EP_{(j)}$.
 \item For $j\leq n$, the restriction to
       $g_j\EP_{(j)}$ is the identity.
\end{itemize}

\begin{lemma}
\label{lem:H_n-construction}
By elementary obstruction theory, we can construct a homotopy 
$H_n\colon \EG\times [0,1] \to \EG$ such that 
\begin{enumerate}
 \item $H_n(p,0) = p$ for all $p\in \EG$;
 \item $H_n(p,1) = \tau_n\circ \theta_n(p)$ for all $p\in \EG$;
 \item \label{item:H_n-close}
       $H_n$ is close to the projection 
       $\EG\times [0,1]\ni (p,t) \mapsto p\in \EG$.
\end{enumerate}
\end{lemma}
For the details of the construction, see the arguments in \cite[Section
3.]{MR1388312}. 


We have a coarse compactification
\begin{align*}
 \barERX_n&\coloneqq \ERX_n \cup  \left\{\dGP 
   \setminus \bigcup_{j\leq n}\{s_j\}\right\} 
   \cup \bigcup_{j\leq n} W_j;\\
\end{align*}

For $n<m$, we extend $\theta_n^m$ to $\bar{\theta}_n^m\colon \barERX_m
\to \barERX_n$ as follows.
\begin{itemize}
 \item The restriction to $\dGP\setminus \bigcup_{j\leq m}\{s_j\}$ is
       the identity.
 \item For $n< j \leq m$, $W_j \ni x\mapsto s_j \in
       \dGP$.
 \item For $1\leq j \leq n$, the restriction to $W_j$ is
       the identity.
\end{itemize}

Since $\{\bar{\theta}_n^m\colon \barERX_m \to \barERX_n\}$ 
forms a projective system, we have a compactification
\begin{align*}
 \barEG& \coloneqq \varprojlim \overline{\ERX_n} = 
   \EG \cup \left\{\dGP \setminus \bigcup_{j\in \N}\{s_j\} \right\}
   \cup \bigcup_{j\in \N} W_j.
\end{align*}
It follows from \cite[Proposition 7.14.]{boundary} that $\barEG$ is a
coarse compactification.
Let $\bar{\theta}_n \colon \barEG \to
\barERX_n$ be the natural projection, which is the canonical extension
of $\theta_n$.


\begin{definition}
\label{def:beta_n-i}
 For $i\leq n$, we define a map
 $\beta_{n}[i]\colon \barERX_n \to \barERX[i]$
 as follows.
 \begin{itemize}
 \item $\Rip(\VaugIR{i}) \ni x \mapsto x\in \Rip(\VaugIR{i})$.
 \item For $j > n$, 
       $\Rip(V[j])_r^R\times [0,1] \ni (x,t)
       \mapsto x\in \Rip(\VaugIR{i})$.
 \item For $j > n$, $\Rip(V[j])_r\times \{1\} \ni (x,1) \mapsto x\in
       \Rip(\VaugIR{i})$.
 \item For $j \leq n, j\neq i$, $\Rip(V[j])_r^R\times [0,1) \ni (x,t) \mapsto
       h_j(x,t)\in \Rip(\VaugIR{i})$.
 \item For $j \leq n, j\neq i$, 
       $g_j\EP_{(j)} \ni x\mapsto f_j^{-1}(x) \in \Rip(\VaugIR{i})$.
 \item The restriction to $\Rip(V[i])_r^R
       \times [0,1] \cup_{f_i} g_i\EP_{(i)}$ is the identity.
 \item The restriction to $\dGP \setminus
       \bigcup_{j \leq n}\{s_j\}$ is the identity.
 \item For $j \leq n, j\neq i$, 
       $W_j \ni x \mapsto s_j \in \dGP$.
 \item The restriction to $W_i$ is the identity.
 \end{itemize}
\end{definition}

\begin{lemma}
 If $(\overline{\EP_i},\EP_i)$ satisfies the condition \BM{} for all
 $1\leq i \leq k$, then so does $(\barERX_n,\ERX_n)$ for all $n\in \N$.
\end{lemma}

\begin{proof}
We will show that the
triplet $(\barERX_n,\ERX_n,x)$ satisfies local-\BM{} for all points
$x\in \barERX_n\setminus \ERX_n$.

First, we suppose that $x\in W_i$ for $1\leq i\leq n$. Then we can find a
 neighbourhood $U_{x,i}$ of $x$ in $\barERX_n$ such that

\begin{align*}
 U_{x,i} \cap \left\{\bigcup_{1\leq j\leq n,\, j\neq i}
 \Rip(V[j])_r^R)\times[0,1]
  \cup_{f_j} g_j\EP_{(j)} \right\} = \emptyset.
\end{align*}
 We can construct the right inverse
 $\beta_n^{-1}[i]\colon \beta_n[i](U_{x,i})\to \barERX_n$ of $\beta_n[i]$,
 and a homotopy $h\colon \ERX_n\times [0,1] \to \ERX_n$
 satisfying the conditions in
 Lemma~\ref{lem:hmtpymap-local-BM}. By Proposition~\ref{prop:BM-single},
 the triplet $(\barERX[i],\ERX[i],f(x))$ satisfies local-\BM{}, thus so
 does $(\barERX_n,\ERX_n,x)$.
 
Next, we suppose that $x\in \dGP\setminus \bigcup_{j\leq n}\{s_j\}$.  In
this case, we can also show that the triplet $(\barERX_n,\ERX_n,x)$
satisfies local \BM{}, by the argument similar to the first case with
replacing $\beta_n[j]$ by $\beta_n[1]$.
\end{proof}

\begin{lemma}
\label{lem:tau_n-close} Let $x\in \barERX_n \setminus \ERX_n$ be a point
 and let $U_0$ be a neighbourhood of $x$ in $\barERX_n$. Then there
 exists a neighbourhood $U_1$ of $x$ in $\barERX_n$ such that
\begin{align}
\label{eq:tau_n-close}
 \tau_n(U_1\cap \ERX_n) \subset \bar{\theta}_n^{-1}(U_0) \cap \EG.
\end{align}
\end{lemma}

\begin{proof}
 Let $x\in \barERX_n \setminus \ERX_n$ be a point and let $U_0$ be a
neighbourhood of $x$ in $\barERX_n$. 

First, we consider the case that $x\in W_i$ for $i\leq n$.  There exists
a neighbourhood $U_0'$ of $x$ in $g_iP_{(i)}\cup W_i$ such that
$\RDp(U_0') \subset U_0$ and for all $l\leq n, \, l\neq i$, we have
\begin{align*}
 \RDp(U_0')\cap
   \left(\Rip(V[l])_r^R\times[0,1]\cup_{f_l} g_l\EP_{(l)}\right) 
 = \emptyset.
\end{align*}
Since the restriction of $\tau_n$ to 
$\Rip(\Vaug)^R \cup \Rip(V[i])_r^R\times[0,1]\cup_{f_i} g_i\EP_{(i)}$
is the identity, we have 
\begin{align*}
 \tau_n\left(U_0\cap 
\left\{\Rip(\Vaug)^R\cup \Rip(V[i])_r^R\times[0,1]
 \cup_{f_i} g_i\EP_{(i)}\right\}\right) 
\subset \bar{\theta}_n^{-1}(U_0).
\end{align*}
For $j>n$, set 
\begin{align*}
 RV[j]: = \Rip(V[j])_r^R\times [0,1] \cup \Rip(V[j])_r\times \{1\}.
\end{align*}
We remark that
\begin{align*}
\tau_n(RV[j])\subset \Rip(V[j])_r^R\times[0,1]\cup_{f_j} g_j\EP_{(j)}.
\end{align*}
Set $U_1\coloneqq\RDp(U_0')$. 
By (\ref{item:FinvFUo}) in \cref{sec:rips-compl-augm},
it is easy to show that for $j>n$, if $U_1\cap RV[j] \neq \emptyset$,
then $RV[j]\subset U_1$, and thus
\begin{align*}
 \Rip(V[j])_r^R\times[0,1]\cup_{f_j} g_j\EP_{(j)}
 \subset \bar{\theta}_n^{-1}(U_1) \cap \EG.
\end{align*}
It follows that, if $U_1\cap RV[j] \neq \emptyset$, then
\begin{align*}
 \tau_n(U_1\cap RV[j]) \subset \tau_n(RV[j]) 
 \subset \bar{\theta}_n^{-1}(U_1) \cap \EG.
\end{align*} 
Therefore we have $\tau_n(U_1)\subset \bar{\theta}_n^{-1}(U_1)\cap \EG$.

In the case $x\in \dGP \setminus \bigcup_{j\leq n}\{s_j\}$,
we can prove (\ref{eq:tau_n-close}) using the same argument with replacing 
$\RDp(U_1)$ by $\RDc(U_1)$.
\end{proof}

\begin{lemma}
\label{lem:H_n-close} Let $x\in \barERX_n \setminus \ERX_n$ be a
 point and let $U_0$ be a neighbourhood of $x$ in $\barERX_n$. Then
 there exists a neighbourhood $U_1'$ of $x$ in $\barERX_n$ such that
\begin{align}
\label{eq:H_n-close-U-contain}
 H_n(\theta_n^{-1}(U_1'\cap \ERX_n)\times [0,1]) 
 \subset \theta_n^{-1}(U_0)\cap \EG. 
\end{align}
\end{lemma}

\begin{proof}
 Let $x\in \barERX_n \setminus \ERX_n$ be a point and let $U_0$ be a
 neighbourhood of $x$ in $\barERX_n$. By the
 condition~(\ref{item:H_n-close}) in Lemma~\ref{lem:H_n-construction},
 we can find a smaller neighbourhood $U_1'$ of $x$ in $\barERX_n$
 satisfying the equation (\ref{eq:H_n-close-U-contain}).
\end{proof}

\begin{proposition}
\label{prop:EG-BM}
 If $(\overline{\EP_i},\EP_i)$ satisfies the condition \BM{} for all
 $1\leq i \leq k$, then so does $(\barEG,\EG)$.
\end{proposition}

\begin{proof}
Let $x\in \barEG\setminus \EG$ be a point. We will show that the
triplet $(\barEG,\EG,x)$ satisfies local-\BM{}.

Let $U$ be a neighbourhood of $x$ in $\barEG$. 
Since $\barEG$ is the projective limit of $\barERX_n$, 
there exist $n \in \N$ and 
a neighbourhood $U_0$ of $x$ in $\barERX_n$ such that 
$\bar{\theta}_n^{-1}(U_0)\subset U$.

By Lemma~\ref{lem:tau_n-close} and Lemma~\ref{lem:H_n-close},
there exists a neighbourhood $U_1\subset U_0$ of $x$ in $\barERX_n$ such that 
$\tau_n(U_1)$ is a neighbourhood of $x=\tau_n(x)$ in $\barEG$, and
\begin{align*}
 &\tau_n(U_1\cap \ERX_n) \subset U\cap \EG,\\
 &H_n(\theta_n^{-1}(U_1\cap \ERX_n)\times [0,1]) \subset U\cap \EG,
\end{align*} 
where $H_n$ is the homotopy constructed in \cref{lem:H_n-construction}.
Since the pair $(\barERX_n,\ERX_n)$ satisfies the condition \BM{}, 
by \cref{lem:map-local-BM}, the triplet $(\barEG,\EG,x)$ satisfies local-\BM{}.

%
%
%
\end{proof}

\subsection{Proof of Theorem~\ref{thmIntro:main-BM}}
\begin{lemma}
\label{lem:blown-up-finite-dim} Let $G$ be a finitely generated group
 which is hyperbolic relative to a finite family of infinite subgroups
 $\famP=\{P_1,\dots, P_k\}$.  We suppose that each subgroup $P_i$ admits
 a finite $P_i$-simplicial complex $\EP_i$ which is a universal space
 for proper actions. For $i=1,\dots,k$, let $\overline{\EP_i}$ be a
 coarse compactification of $\EP_i$.  Set $W_i\coloneqq\overline{\EP_i}
 \setminus \EP_i$. Let $\dG$ be the blown-up corona of $(G,\famP,
 \{W_1,\dots, W_k\})$ and let $\barEG\coloneqq\EG \cup \dG$ be a coarse
 compactification.  Suppose that for all $i = 1,\dots,k$,
 $\overline{\EP_i}$ has finite covering dimension, then so does $\barEG$.
\end{lemma}
\begin{proof}
 Since $\EG$ has finite dimension, it is enough to show that 
the blown-up corona
\[
 \dG= \left(\dGP \setminus \bigcup_{i\in
\N}\{s_i\}\right) \cup \bigcup_{i\in \N}W_i
\] has a finite dimension. Here we remark that $W_i = g_iW_{(i)}$ for $i\in \N$.

 We denote by $\hat{\Gamma}(G)$ the {\itshape coned-off} Cayley graph,
 which was introduced by Farb\cite{Far98}. We remark that
 $\hat{\Gamma}(G)$ is hyperbolic geodesic space but is not proper
 metric space. It is known that its Gromov boundary $\partial
 \hat{\Gamma}(G)$ is homeomorphic to $\dGP \setminus
 \bigcup_{i\in\N}\{s_i\}$. See \cite[Proposition 9.1.]{MR2922380}.

 In the proof of~\cite[Lemma 3.7.]{MR1974394}, it is shown that
 $\partial \hat{\Gamma}(G)$ has finite dimension. Since
 $\bigcup_{i\in \N}W_i$ is a countable union of bounded dimensional
 compact sets, by~\cite[Theorem III 2.]{MR0006493}, it has finite
 dimension. Then it follows from \cite[III-2-B]{MR0006493} that
 \[
  \dim \dG \leq \dim (\partial \hat{\Gamma}(G)) + \dim(\cup_i W_i) + 1.
 \]
 Therefore $\dG$ has a finite dimension.
\end{proof}

\begin{proof}[Proof of Theorem~\ref{thmIntro:main-BM}]
A universal space $\EG$ is constructed in Section~\ref{sec:projective-limit}.
 By Proposition~\ref{prop:EG-BM} and
 Lemma~\ref{lem:blown-up-finite-dim}, the pair $(\barEG,\EG)$ satisfies
 the condition \fBM{}. Thus by Proposition~\ref{prop:Z-set},
 $\dG$ is a $\mathcal{Z}$-set in $\barEG$. So $\barEG$ is homotopic to $\EG$.
Since $\EG$ is contractible, so is $\barEG$.
This proves the statement~(\ref{item:BM}). 
The statement~(\ref{item:equiv}) follows
 from Proposition~\ref{prop:equivariant-corona}.
\end{proof}

\section{cohomological dimension}
\label{sec:cohom-dimens}


\subsection{Proof of Corollary~\ref{thm:main-cohomdim}}
Let $G$ be a finitely generated
 group that is hyperbolic relative to a finite family of infinite
 subgroups $\famP=\{P_1,\dots, P_k\}$.  We suppose that each subgroup
 $P_i$ admits a finite $P_i$-simplicial complex $\EP_i$ which is a
 universal space for proper actions.  For $i=1,\dots,k$, let
 $\overline{\EP_i}$ be a coarse compactification of $\EP_i$.  Set
 $W_i\coloneqq\overline{\EP_i} \setminus \EP_i$.  Let $\dG$ be the
 blown-up corona of $(G,\famP, \{W_1,\dots, W_k\})$. 

Let $HX^* (G;\Z)$ denotes the coarse cohomology of $G$. By \cite[Example 5.21]{MR2007488}
we have $H^*(G;\Z G) \cong HX^* (G;\Z)$. 
Then by \cite[Theorem 1.2.]{boundary},
 we have
\begin{align*}
 H^p(G;\Z G) \cong HX^p(G;\Z) \cong \widetilde{H}^{p -1}(\dG;\Z) \quad \text{for all} \quad p\in \N.
\end{align*}

Now we suppose that
 $(\overline{\EP_i},W_i)$ is a \calZ-structure of $P_i$ for all 
 $1\leq i \leq k$.
Then by Theorem~\ref{thmIntro:main-BM}  and
Theorem~\ref{thm:BM-implies-cohomdim}, 
we have

\begin{align*}
  \dim(\dG) =& \max\{p: H^{p}(\dG;\Z)\neq 0\}\\
            =& \max\{p: H^{p}(G;\Z G)\neq 0\} -1.
\end{align*}

Let $\EG$ be a finite $G$-simplicial complex which is a universal space
for proper actions. 
Now we assume that $G$ is torsion free.
Then, the action of $G$ on $\EG$ is free, so 
the quotient $\EG/G$ is a finite simplicial complex, which is a $K(G,1)$-space. 
Thus $G$ is of type FL by \cite[VIII (6.3)]{Brown-cohom-gr}.
Therefore by \cite[VIII (6.7)]{Brown-cohom-gr}, we have

\begin{align}
\label{eq:cd-blbdry}
  \cd{}{G} = \max\{p: H^{p}(G;\Z G) \neq 0\} = \dim(\dG) + 1.
\end{align}
Alternatively, we can also prove \cref{eq:cd-blbdry} by applying 
\cite[Theorem 1.7]{Bestv-local-homo-bdry}

\subsection{Examples}
\label{sec:calc-examples}
\label{sec:example-cohomdim} We give two examples of the explicit
calculation of the dimension of the blown-up coronae. In the first case,
the dimension of the blown-up corona is greater than that of the
Bowditch boundary. In the second case, the opposite inequality holds.
\begin{example}
 Let $A, B$ be finitely generated groups. A free product $A*B$ is
 hyperbolic relative to $\{A, B\}$. The Bowditch boundary
 $\partial(A*B,\{A,B\})$ is homeomorphic to the Cantor space. Therefore
 $\dim \partial(A*B,\{A,B\}) = 0$. Let $n\geq 2$ be an integer.  Set
 $A=B\coloneqq\Z^n$.  Then $\underline{E}\Z^n = \R^n$ and a closed disk $D^n \coloneqq
 \{x\in \R^n: \abs{x} \leq 1\}$ is a coarse compactification of $\R^n$
 with a corona which is homeomorphic to the sphere $S^{n-1}$. 
Let $\partial^{\mathrm{bl}}(\Z^n * \Z^n)$ be the blown-up
 corona of $(\Z^n * \Z^n,\{\Z^n,\Z^n\},\{S^{n-1},S^{n-1}\})$. Then we
 have
\begin{align*}
 \tilde{H}^{p}(\partial^{\mathrm{bl}}(\Z^n * \Z^n)) 
 =\begin{cases}
   \bigoplus_{i\in \N}\Z, & \text{ if } p = n-1\text{ or } p=0,\\
   0, & \text{otherwise}.
  \end{cases}
\end{align*}
 Therefore
 $\dim \partial(A*B,\{A,B\}) =0 < \dim \partial^{\mathrm{bl}}(\Z^n * \Z^n) = n-1$.
 For details of the above computation, see \cite{fukaya2024boundary}.
\end{example}
\begin{proposition}
\label{prop:hypmanifd-cohomdim} Let $G$ be a group which properly
 isometrically acts on an $m$-dimensional simply-connected pinched
 negatively curved complete Riemannian manifold $X$.  Suppose that the
 quotient is with finite volume, but not compact.

 Let $\famP$ be a set of representatives
 of conjugacy invariant classes of maximal parabolic subgroups of $G$
 with respect to the action on $X$. 
 Let $\dG$ be a blown-up corona of $(G,\famP)$ described in the proof of
 \cite[Corollary 9.3.]{boundary}. 
 Then we have
 \begin{align*}
   \dim \dGP = m-1 > \dim \dG = m-2.
 \end{align*}
\end{proposition}

\begin{proof}
 It is known that $\mathbb P$ is a finite family of virtually
nilpotent groups, and $G$ is hyperbolic relative to $\mathbb P$
(\cite[8.6]{asym_invs} and also \cite[Theorem 5.1]{Far98}). The Bowditch
boundary $\dGP$ is the visual boundary of $X$, which is homeomorphic to
the $m-1$ dimensional sphere $S^{m-1}$.
Let $\dG$ be a blown-up corona of $(G,\famP)$ described in the proof of
\cite[Corollary 9.3.]{boundary}. It is shown that 
\begin{align*}
 H^p (G;\Z G)\cong  
 HX^p (G)\cong \tilde{H}^{p -1}(\dG)\cong \left\{
\begin{array}{cc}
\bigoplus_{i\in\mathbb N} \mathbb Z   & (p =m-1)   \\
0  & (p \neq m-1)   \\
\end{array}
\right. 
.
\end{align*}

By Corollary~\ref{thm:main-cohomdim}, we have 
$\cd{}{G} = \dim \dG +1 = m-1$.
\end{proof}

\begin{remark}
 Takatsu \cite[Section 5.3]{takatsu2023blown} 
 obtained further examples of explicit computations of the dimensions of blown-up coronae
 by using the automorphisms groups of K3 surfaces related to sphere packings. 
\end{remark}

\section{Blown-up corona of a hyperbolic group with an almost malnormal family of  quasiconvex subgroups}
\label{sec:hyp-malnormal} Let $G$ be a non-elementary hyperbolic group,
and let $\famP=\{P_1,\dots, P_k\}$ be an almost malnormal family of
proper, quasiconvex subgroups of $G$. By Theorem~\ref{th:hyp-is-relhyp},
$G$ is hyperbolic relative to $\famP$.


Manning~\cite{Manning-Bowditch-boundary} gave an explicit description of
its Bowditch boundary $\dGP$. In this section, using his description, we
show that the blown-up corona of $(G,\famP)$ is homeomorphic to the Gromov
boundary. First, we recall the description of $\dGP$ following~\cite{Manning-Bowditch-boundary}.


Let $\partial G$ be the Gromov boundary of
$G$, and for each $i\in \N$, let $\partial P_i$ be the limit set of
$P_i$.  We see that $g\partial P_i\cap h\partial P_j$ is empty unless
$i=j$ and $g^{-1}h\in P_i$.  We choose an order of coset
$\{g_n\}_{n\in\N}$ of $(G,\famP)$ in the sense of
Definition~\ref{def:order-of-cosets}.  Set
\begin{align}
 A&\coloneqq\{g_i\partial P_{(i)}: i\in \N\}; \label{eq:5}\\
 B&\coloneqq\{\{x\}:x\in \partial G\setminus \bigcup A\}.\label{eq:6}
\end{align}
The union $\mathcal{C}\coloneqq A\cup B$ is therefore a decomposition of
$\partial G$ into closed sets. By Theorem~\ref{th:G-to-B-full}, the quotient
topological space $\partial G/\mathcal{C} = A\cup B$ is homeomorphic to
the Bowditch boundary $\dGP$.

We need the following elementary fact on continuous maps from compact spaces
to Hausdorff spaces.
\begin{lemma}
\label{lem:hX-Y} 
 Let $X$, $Y$, and $Z$ be topological spaces, and 
 let $\xi \colon Z\to X$ be a continuous surjection.
 We suppose that $Z$ is compact and $X$ is Hausdorff.
 Then a map $f\colon X\to Y$ is continuous if and only if
 the composite $f\circ \xi\colon Z\to Y$ is continuous.
\end{lemma}

\begin{proof}
 Under the hypothesis of the lemma, it is well known that $\xi$ is a closed map.
 Suppose that the composite $f\circ \xi\colon Z\to Y$ is continuous.
 Let $A\subset Y$ be a closed set. 
 So $(f\circ \xi)^{-1}(A)$ is closed in $Z$. 
 Then the image $\xi ((f\circ \xi)^{-1}(A))$ is closed in $X$.
 Since $\xi$ is surjective, 
 \begin{align*}
  f^{-1}(A) = \xi (\xi^{-1}(f^{-1}(A))) = \xi ((f\circ \xi)^{-1}(A)).
 \end{align*}
 Thus $f^{-1}(A)$ is closed.  Therefore $f$ is continuous.
\end{proof}



\begin{proof}[Proof of Theorem~\ref{th:B-to-G}]
 We use the notation defined by (\ref{eq:5}) and (\ref{eq:6}).  By the
 construction of $\dG$ in \cite[Definition 7.14]{boundary}, we have a
 decomposition
 \begin{align*}
  \dG = \left(\dGP\setminus \cup_i \{s_i\}\right)
  \cup \left(\cup_i g_i\partial P_{(i)}\right).
 \end{align*}
 By Theorem~\ref{th:G-to-B-full}, we can identify
 $\dGP\setminus \cup_i \{s_i\}$ with the subset $B$.
 Thus, there exists a canonical
 $G$-equivariant bijection $\varphi\colon G\cup \dG \to G\cup \partial G$.
 Since both of $\dG$ and $\partial G$ are coarse compactifications,
 there exist continuous surjections $\xi_\infty\colon hG \to G\cup \dG$
 and $\eta\colon hG \to G\cup \partial G$, both are the extensions of the
 identity $G\to G$, where $hG$ is the Higson corona of $G$.  We
 remark that an explicit description of $\xi_\infty$ is given in
 \cite[Section 7.2.]{boundary}. It follows from this description of
 $\xi_\infty$, we see that 
 $\eta = \varphi\circ \xi_\infty$
Thus
 $\varphi$ is continuous by Lemma~\ref{lem:hX-Y}, therefore it is 
a homeomorphism.
\end{proof}

\section{Action on the blown-up corona}
\label{sec:act-on-corona}
In this section, we show the following proposition, which corresponds to
the statement~(\ref{item:equiv}) in Theorem~\ref{thmIntro:main-BM}.

\begin{proposition}
\label{prop:equivariant-corona}
 Let $G$ be a group that is hyperbolic relative to a finite family of
infinite subgroups $\famP = \{P_1,\dots, P_k\}$. Suppose that for all 
$i\in \{1,\dots,k\}$, each $P_i$ has $P_i$-equivariant corona 
 $\partial P_i$. Then the blown-up corona of 
$(G,\famP,\{\partial P_1,\dots\partial P_k\})$ is an $G$-equivariant corona.
\end{proposition}

\subsection{Equivariant corona}
\label{sec:equivariant-corona} Let $X$ be a proper metric space. We
denote by $hX$ the Higson compactification of $X$, and denote by $\nu
X \coloneqq hX\setminus X$ the Higson corona.  A corona $X$ is a pair
$(W,\zeta)$ of a compact Hausdorff space $W$ and a continuous map
$\zeta\colon \nu X\to W$.  A corona $(W,\zeta)$ gives a coarse
compactification $\XbarW\coloneqq hX\cup_\zeta W$, which is obtained by
identifying $x\in \nu X$ with $\zeta(x)\in W$. See \cite[Section
2.2]{boundary} for details.

Let $G$ be a group acting on $X$ by isometries.  A corona $(W,\zeta)$ is
{\itshape $G$-equivariant} if $G$ acts continuously on $\XbarW$ such
that the embedding $X\hookrightarrow \XbarW$ is $G$-equivariant.  We
also say that $\XbarW$ is a {\itshape $G$-equivariant compactification}.




\begin{lemma}
\label{lem:zeta-equivariant}
Let $X$ be a proper metric space, and $G$ be a group acting on $X$ by 
 isometries. Let $(W,\zeta)$ be a corona of $X$. Assume that $G$ also 
 acts on $W$. Then the corona $(W,\zeta)$ is $G$-equivariant if and only if
 so is the map $\zeta \colon \nu X \to W$.
\end{lemma}

\begin{proof}
We denote, respectively, by $\alpha^\nu$ and $\alpha$ the action of $G$ 
 on $\nu X$ and $W$. 
We also denote by $\bar{\zeta}$ an extension of $\zeta$ to $hX$ 
by the identity,  that is, 
$\bar{\zeta}\coloneqq \mathrm{id}\cup \zeta \colon hX\to \XbarW$.

We suppose that $\zeta$ is $G$-equivariant. We note that $\XbarW$ has 
a quotient topology induced by $\bar{\zeta}$. Since the action $\alpha^\nu$ 
is continuous, so is the action $\alpha$.
Thus $(W,\zeta)$ is $G$-equivariant.
 
Conversely, we suppose that a corona $(W,\zeta)$ is $G$-equivariant.
 Let $x^\nu \in \nu X$. We choose a sequence 
 $\{x_\lambda\}_{\lambda\in \Lambda}$ in $X$ which converges to $x^\nu$. 
 Set $x\coloneqq \zeta(x^\nu)$. 
Then $x_\lambda$ converges to $x$ in $\XbarW$, 
 since $\bar{\zeta} \colon hX\to \XbarW$ is continuous.
Now let $g\in G$. 
Since the action $\alpha^\nu$ is continuous, 
 $g\cdot x_\lambda$ converges to $\alpha^\nu(g)\cdot x$ in $hX$.
 By the same reason, $g\cdot x_\lambda$ converges to $\alpha(g)\cdot x$ in 
 $\XbarW$. Therefore we have 
 $\zeta(\alpha^\nu (g)\cdot x^\nu) = \alpha(g)\cdot x$.
\end{proof}

\subsection{Action on the blown-up corona}
\label{sec:action-blown-up}
 Let $G$ be a group that is hyperbolic relative to a finite family of
infinite subgroups $\famP = \{P_1,\dots, P_k\}$. Suppose that for all
$i\in \{1,\dots,k\}$, each $P_i$ has $P_i$-equivariant corona $\partial
P_i$. We set $\overline{P}_i\coloneqq P_i \cup \partial P_i$.

We define isometries $\ciso_i \colon P_{(i)} \to g_iP_{(i)}$ by left
multiplication of $g_i$, that is, $\ciso_i(p) \coloneqq g_ip$ for 
$p\in P_{(i)}$. Then we have canonical continuous extensions 
$\bar{\ciso}_i \colon \overline{P}_{(i)} \to g_i\overline{P}_{(i)}$.

For each $i\in \N$, we define maps $\phi_i\colon G \to \N$ and 
$\psi_i\colon G \to P_{(i)}$ as follows.
For any $h\in G$, there exists a unique integer $\phi_i(h)$ such that 
\[
 hg_iP_{(i)} = g_{\phi_i(h)}P_{(i)}.
\]
Then we set $\psi_i(h)\coloneqq g_{\phi_i(h)}^{-1}hg_i$.

\begin{remark}
\label{rem:phi-modulo}
For $i\in \N$ and $h\in G$, we have that $\psi_i(h)$ lies on $P_{(i)}$ and
$(i) = (\phi_i(h))$.
\end{remark}

We have the following commutative diagram.
\begin{align*}
 \xymatrix{
g_iP_{(i)}\ar[r]^{h\cdot} \ar[d]_{\ciso_i^{-1} = g_i^{-1}\cdot}
 & g_{\phi_i(h)}P_{(i)}\\
P_{(i)}\ar[r]^{\psi_i(h)\cdot}
 & P_{(i)}  
 \ar[u]_{\ciso_{\phi_i(h)} = g_{\phi_i(h)}\cdot} 
}
\end{align*}
Here, upper and lower horizontal arrows are multiplications of 
$h$ and $\psi_i(h)$, respectively, from the left.

\begin{lemma}
\label{lem:psi-homlike}
 Let $i \in \N$ and $h,k\in G$. Set $j\coloneqq \phi_i(k)$. Then we have 
 $\phi_j(h) = \phi_i(hk)$.
\end{lemma}

\begin{proof}
 By Remark~\ref{rem:phi-modulo}, both of $g_{\phi_j(h)}^{-1}hg_j$ 
 and $g_{\phi_i(k)}^{-1}kg_i$ lie on $P_{(j)} = P_{(i)}$. So we have
 \[
  g_{\phi_j(h)}^{-1}hkg_i = g_{\phi_j(h)}^{-1}hg_j g_j^{-1}kg_i 
 \in P_{(i)} = P_{(j)}.
 \]
 Thus we have
 \[
  hkg_i \in g_{\phi_j(h)}P_{(j)}\cap g_{\phi_i(hk)}P_{(i)} \neq \emptyset.
 \]
 Therefore $\phi_j(h) = \phi_i(hk)$ by \cref{rem:order-ij}.
\end{proof}


Let $\dG$ be the blown-up corona defined in 
\cite[Definition 7.13.]{boundary}.
Let $\pi_n \colon \dX_{n+1} \to \dX_{n}$ 
be the projection defined in the proof of 
~\cite[Proposition 7.14.]{boundary}.
Set $\pi_0^n\coloneqq \pi_0\circ \dots \circ \pi_n$. 
Then we have a continuous map 
$\pi\coloneqq \varprojlim \pi_0^n \colon \dG \to \dGP$.

For a parabolic points $s_i$ of $g_iP_{(i)}$ in $\dGP$, we have
$\pi^{-1}(s_i) = \partial (g_iP_{(i)})$. 
We set $\Omega \coloneqq \bigcup_{i\in \N} \partial(g_iP_{(i)})$ and 
$\Lambda\coloneqq \dG \setminus \Omega$.  
Then $\Omega = \pi^{-1}(\{s_i:i\in \N\})$ and 
the restriction of $\pi$ to $\Lambda$ is a bijection onto 
$\Lambda_0 \coloneqq \dGP \setminus \{p_i:i\in \N\}$, 
thus we identify $\Lambda$ with $\Lambda_0$.

Now we define a group action $\alpha$ of $G$ 
on $\dG$. Let $\alpha_0$ be the action of $G$ on $\dGP$.
We fix $h\in G$. For $x\in \Lambda$, we define
$\alpha(h) \cdot x\coloneqq \alpha_0(h) \cdot x$.
Next, for $i\in \N$ and $x\in \partial(g_iP_{(i)}) \subset \Omega$, 
we define 
\begin{align*}
 \alpha(h)\cdot x\coloneqq 
   \bar{\ciso}_{\phi_i(h)}(\psi_i(h)\cdot \bar{\ciso}_i^{-1}(x)).
\end{align*}
Here $\psi_i(h)\cdot \bar{\ciso}_i^{-1}(x)$ is an action of 
$\psi_i(h) \in P_{(i)}$ to a point $\bar{\ciso}_i^{-1}(x)$ 
in $\partial P_{(i)}$.

\begin{lemma}
 For any $h,k \in G$ and $x\in \dG$,
 we have $\alpha(h)\cdot (\alpha(k)\cdot x) = \alpha(hk)\cdot x$.
\end{lemma}
\begin{proof}
 Let $h,k \in G$. Let $x\in \dG$. 
 First, we suppose $x\in \Lambda$. 
 Since $\alpha(k)\cdot x = \alpha_0(k)\cdot x \in \Lambda_0 = \Lambda$, 
 we have $\alpha(h)\cdot (\alpha(k)\cdot x) = \alpha(hk)\cdot x$. 

Next, we suppose $x\in \partial(g_iP_{(i)})$ for some $i\in \N$.
Set $j\coloneqq\phi_i(k)$. Then by Remark~\ref{rem:phi-modulo}, we have $(j) = (i)$.
Now we have
\begin{align}
 \alpha(h)\cdot (\alpha(k)\cdot x)
 &= \alpha(h)\cdot 
    (\bar{\ciso}_{\phi_i(k)}\circ \psi_i(k)\circ \bar{\ciso}_i^{-1}(x))
    \nonumber\\
 &= \bar{\ciso}_{\phi_j(h)}\circ \psi_j(h)\circ \bar{\ciso}_j^{-1}
    \circ \bar{\ciso}_{j}\circ \psi_i(k)\circ \bar{\ciso}_i^{-1}(x)
    \nonumber\\
 &= \bar{\ciso}_{\phi_j(h)}\circ \psi_j(h)\circ 
    \psi_i(k)\circ \bar{\ciso}_i^{-1}(x)
    \nonumber\\
 &= \bar{\ciso}_{\phi_j(h)}\circ 
 (g_{\phi_j(h)}^{-1}hg_{j} g_{j}^{-1}kg_i) \circ \bar{\ciso}_i^{-1}(x)
   \nonumber\\
 &= \bar{\ciso}_{\phi_j(h)}\circ 
 (g_{\phi_j(h)}^{-1}hkg_i) \circ \bar{\ciso}_i^{-1}(x)
   \nonumber\\
 &= \bar{\ciso}_{\phi_i(hk)}\circ \label{eq:a}
 (g_{\phi_i(hk)}^{-1}(hk)g_i) \circ \bar{\ciso}_i^{-1}(x)
   \nonumber\\
 &= \alpha(hk)\cdot x. \nonumber
\end{align}
 Here, the second equality from below holds by 
 Lemma~\ref{lem:psi-homlike}. 
\end{proof}

Now we have shown that the action $\alpha$ is well-defined.  We will
show that it is continuous.  By~\cite[Proposition 7.14.]{boundary},
there exists a canonical surjection 
$\xi_\infty\colon \nu G \to \dG$. 
For simplicity, we abbreviate $\xi_\infty$ by $\xi$.
\begin{lemma}
\label{lem:commute}
The map $\xi$ is $G$-equivariant.
\end{lemma}

\begin{proof}
 Let $\xi_0\colon \nu \Xaug \to \dGP$ be a map defined 
 in \cite[Section 7.2]{boundary}.
Since $\xi_0$ is $G$-equivariant,
The restriction of $\xi$ on $\xi^{-1}(\Lambda)$ is $G$-equivariant.
%
%
We suppose $h\in G$ and $x\in \xi^{-1}(\Omega)$. Then we have 
$x\in \xi^{-1}(\partial (g_iP_{(i)}))$ for some $i\in \N$.
Set $j\coloneqq \phi_i(h)$. 
We denote by $\alpha^\nu$ the action of $G$ on the Higson corona $\nu G$.
The following diagram is commutative.
\begin{align*}
 \xymatrix{
 & \ar[dl]_\xi \xi^{-1}( \partial (g_iP_{(i)})) \ar[r]^{\alpha^\nu(h)} 
   \ar[d]_{\nu \ciso_i^{-1}} & \xi^{-1}(\partial (g_{j}P_{(i)}))\ar[rd]^\xi &\\
 \ar[dr]^{\bar{\ciso}_i^{-1}} \partial (g_iP_{(i)})& 
   \xi^{-1}(\partial P_{(i)}) \ar[d]_\xi \ar[r]^{\nu \psi_i(h)}&  
   \xi^{-1}(\partial P_{(i)}) \ar[d]_\xi \ar[u]^{\nu \ciso_j}& 
   \partial (g_jP_{(j)})&\\
 & \partial P_{(i)} \ar[r]^{\psi_i(h)} 
 & \partial P_{(i)} \ar[ru]^{\bar{\ciso}_j} &
 }
\end{align*}
This shows that $\alpha(h)(\xi(x)) = \xi(\alpha^\nu(h)(x))$.
\end{proof}

\begin{corollary}
 \label{cor:alpha-conti}
 For $h\in G$, the map $\alpha(h)$ is continuous.
\end{corollary}
\begin{proof}
 By \cref{lem:commute}, $\alpha(h)\circ\xi$ is continuous, 
 so by \cref{lem:hX-Y}, $\alpha(h)$ is continuous.
\end{proof}

Proposition~\ref{prop:equivariant-corona} follows by 
\cref{lem:zeta-equivariant,lem:commute,cor:alpha-conti}.

\appendix
\crefname{theorem}{Thm}{Thm}
\crefname{lemma}{Lem.}{Lem.}
\crefname{corollary}{Corollary}{Corollaries}
\crefname{proposition}{Proposition}{Propositions}
\crefname{remark}{Rem.}{Rem.}
\crefname{example}{Example}{Examples}
\crefname{section}{Section}{Sections}
\crefname{definition}{Def.}{Def.}
\crefname{equation}{}{}
\crefname{claim}{Claim}{Claims}

\newpage
\section{Table of symbols}
\label{sec:table-symbols}
\begin{table}[h]
 \centering

\begin{tabular}{c|c|c}
\hline
\multicolumn{3}{c}{~\cref{sec:introduction}}\\
\hline
  $\dGP$ & Bowditch boundary & \cref{th:G-to-B-full}\\
\hline
\multicolumn{3}{c}{~\cref{sec:higs-comp}}\\
\hline
 \BM{}& condition related to \calZ-structures & \cref{def:BMconditions}\\ 
 local-\BM{}& localization of \BM{} & \cref{def:BMconditions}\\ 
\hline
\multicolumn{3}{c}{~\cref{sec:relat-hyperb-groups}}\\
\hline
  $\Horo(P)$ & combinatorial horoball  & \cref{def:Horoball}\\
  $\Xaug $ & augmented space & \cref{notation:1}\\
  $(i,g,l)$ & coordinate on $\Horo(g_iP_{(i)})$ & \cref{rem:vertices}\\
\hline
\multicolumn{3}{c}{~\cref{subsec:setting}}\\
\hline
  $\Vaug $& vertex set of $\Xaug$ & \cref{def:notationA}\\
  $V(\EP)$& vertex set of $\EP$ & \cref{def:notationA}\\
  $\Delta_P,e_P$& fundamental domain of $P$ on $\EP$ & \cref{def:notationA}\\
  $e_P$& base point $e_P\in \Delta_P$& \cref{def:notationA}\\
  $\iota_P$& map $P\to V(\EP)$ $g\mapsto g\cdot e_P$& \cref{def:notationA}\\
  $V_P$& map $\EP\to P$ $q\mapsto g$ s.t. $q\in g\Delta_P$ & \cref{def:notationA}\\
  $E_P(K^{0})$& $\cup_{g\in K^{0}} \, (g\cdot\Delta_P)$& \cref{def:notationA}\\
  $s_P$& parabolic point of $\Horo(P)$ & \cref{def:notationA}\\
  $W_P$& corona of $P$ & \\ 

  $\iota_i$ & conjugate of $\iota_{P_{(i)}}$ by $g_i$ & \cref{def:notationB}\\
  $V_i$ & conjugate of $V_P$ by $g_i$ & \cref{def:notationB}\\
  $E_i(K^{0})$ & $\cup_{g\in K^{0}} \, (g\cdot(\Delta_{P_{(i)}}))$ & \cref{def:notationB}\\
  $s_i$& parabolic point of $\Horo(g_iP_{(i)})$ & \cref{def:notationB}\\
  $d_i$ & the restriction of $d_G$ on $g_iP_{(i)}$& \\
\hline
 \end{tabular}
\vspace{4cm}
\end{table}

\begin{table}[h]
\centering
  \begin{tabular}{c|c|c}
\hline
  \multicolumn{3}{c}{~\cref{sec:rips-compl-augm}}\\
 \hline
  $\Vaug_r$ & $\bigsqcup_{i\in \N} (g_iP_{(i)}\times \{n\in \N: n\geq r\})$& \\
  $\Vaug^R$ & $G\sqcup \bigsqcup_{i\in \N} 
             (g_iP_{(i)}\times \{n\in \N: 1\leq n \leq R\});$ & \\ 
  $\Vaug^R$& $\Vaug_r \cap \Vaug^R$& \\ 
  $\Rip(\Vaug)_r$ & full subcomplex of $\Vaug_r$& \\ 
  $\Rip(\Vaug)^R$ & full subcomplex of $\Vaug_R$& \\ 
  $\Rip(\Vaug)_r^R$ & full subcomplex of $\Vaug_r^R$& \\
  $f_P$ & $P$-homotopy equivalence from $\Rip^{(n)}(V(P,\{P\},d_P))_r$ to $\EP$& \cref{eqdef:f_P}\\ 
  $f_i$ & conjugate of $f_P$ by $g_i$ & \cref{def:fi}\\ 
  $V[i]$ & abbreviation of $V(g_iP_{(i)},\{g_iP_{(i)}\},d_i)$& \\ 
  $\VaugIR{i}$ & $ \Vaug \setminus \{(i,g,l):g\in g_iP_{(i)}, l>R\}$ & \cref{eqdef:VaugIRi}\\ 
  $\Rip(\VaugIR{i})$ & full subcomplex of $\VaugIR{i}$ & \\ 
  $\ERX[i]$ & $\Rip(\VaugIR{i}) \cup  \left\{\Rip(V[i])_r^R\times[0,1] \cup_{f_i} g_i\EP_{(i)} \right\}$& \cref{def:ERX}\\
  $V_{[i]}$& map $\ERX[i]\to \VaugIR{i}$ & \cref{eqdef:mapV_bib}\\
  $\partial X[i]$ & blown-up corona of $X[i]$ & \cref{eqdef:dXbib}\\
  $\hat{P}_i$ & $\dGP \setminus \{s_i\}$ & \\
  $\RDc(U)$& neighbourhood of a conical limit point& \\
  $\RDp(U)$& neighbourhood of a parabolic point& \\
 \hline
  \multicolumn{3}{c}{~\cref{sec:proof-theor-refthm:b}}\\
 \hline
  $N(i,L,\epsilon)$& $\epsilon$-neighbourhood of $L$ in $g_iP_{(i)}$& \\
  $\EN(i,L,\epsilon)$& $\bigcup_{g\in N(i,L,\epsilon)} g\Delta_{P_{(i)}}$ & \cref{eqdef:EN}\\
\hline
  \multicolumn{3}{c}{~\cref{sec:projective-limit}}\\
 \hline
 $\ERX_n$ & $\Rip(\Vaug)^R \cup 
 \bigcup_{j=1}^{n}
  \left\{\Rip(V[j])_r^R)\times[0,1]
  \cup_{f_j} g_j\EP_{(j)} \right\}$ & \cref{eqdef:ERXn}\\ 
   & $\cup
\bigcup_{j=n+1}^{\infty}
  \left\{\Rip(V[j]_r^R)\times[0,1]
  \cup \Rip(V[j])_r \times\{1\}\right\} $
& \\ 
  $\EG$ & $\ERX_\infty$& \\ 
 $h_j$  & homotopy from the identity to $f_j^{-1}\circ f_j$& \cref{eqdef:h_j}\\
 $\theta_n^m$  & map $\ERX_m\to \ERX_n$& \\ 
  $\theta_n$   &  $\theta_n^\infty$ & \\ 
 $\tau_n$ & homotopy inverse of $\theta_n$& \cref{eqdef:tau_n}\\ 
 $\beta_{n}[i]$  & map $\barERX_n \to \barERX[i]$ & \cref{def:beta_n-i}\\ 
 \end{tabular}

\end{table}

\newpage
\phantom{dummy}
\bibliographystyle{amsplain}
\bibliography{/Users/tomo/Library/tex/math}

\bigskip
\address{ Tomohiro Fukaya \endgraf
Department of Mathematics and Information Sciences, 
Tokyo Metropolitan University,
Minami-osawa Hachioji, Tokyo 1920397, Japan}

\textit{E-mail address}: \texttt{tmhr@tmu.ac.jp}

\end{document}